\documentclass[12pt]{article}%
\usepackage{amsmath,amssymb,amsthm,amsfonts}
\usepackage{wasysym}
\usepackage{subcaption}
\usepackage{graphicx}
\usepackage[colorlinks]{hyperref}
\usepackage{color}
\usepackage{geometry}
\usepackage{appendix}
\usepackage{textcomp}
\usepackage{stackengine}

\usepackage[export]{adjustbox}

\newtheorem{thm}{Theorem}
\theoremstyle{definition}
\newtheorem{definition}{Definition}
\newtheorem{exmp}{Example}

\newenvironment{customthm}[1]
  {\renewcommand\thethm{#1}\thm}
  {\endthm}

\begin{document}

\title{The One-dimensional Version of Peixoto's Structural Stability Theorem:\\
A Calculus-based Proof}
\author{Aminur Rahman\thanks{Corresponding Author, \url{arahman2@uw.edu}}
\thanks{Department of Applied Mathematics, University of Washington, Seattle, WA 98004},
Denis Blackmore\thanks{Department of Mathematical Sciences, New Jersey Institute of Technology, Newark, NJ 07102}}
\date{}
\maketitle

\section{Summary}

Peixoto's structural stability and density theorems represent milestones in the modern theory of dynamical systems and their applications. Despite the importance of these theorems, they are often treated rather superficially, if at all, in upper level undergraduate courses on dynamical systems or differential equations. This is mainly because of the depth and length of the proofs. In this module, we formulate and prove the one-dimensional analogs of Peixoto's theorems in an intuitive and fairly simple way using only concepts and results that for the most part should be familiar to upper level undergraduate students in the mathematical sciences or related fields. The intention is to provide students who may be interested in further study in dynamical systems with an accessible one-dimensional treatment of structural stability theory that should help make Peixoto's theorems and their more recent generalizations easier to appreciate and understand.  Further, we believe it is important and interesting for students to know the historical context of these discoveries since the mathematics was not done in isolation.  The historical context is perhaps even more appropriate as it is the 100\textsuperscript{th} anniversary of Mar\'{i}lia Chaves Peixoto and Mauricio Matos Peixoto's births, February 24\textsuperscript{th} and April 15\textsuperscript{th} 1921, respectively.

\section{Introduction}

The mathematical foundations of the study of dynamical systems were developed by Leibniz, Newton, and the Bernoulis in the late 1600s and early 1700s, through intimate connections to real world problems with works like ``Nova Methodus'' \cite{Leibniz}, ``Principia'' \cite{Principia}, ``Methodus Fluxionum'' \cite{Methodus}, ``Explicationes'' \cite{Bernoulli}, and many others.  Two centuries later, Poincar\'{e} applied new mathematical techniques to the study of celestial mechanics \cite{Poincare1, Poincare2}, which came to be known as \emph{Dynamical Systems}.  He showed that information about a system can be extracted through its qualitative properties (\emph{phase space}); that is, without having to solve the system of equations.  This is at the heart of dynamical systems theory, and is often accomplished through careful mathematical analysis of differential equations.

During the early years of their mathematics education, students learn to calculate results without having to think about the assumptions being made.  Are these calculations even correct? As we progress in our education we learn that we must first logically show from the assumptions that the calculations are valid.  An early example of this arises in a first course on numerical analysis.  It is easy to devise an algorithm to solve a differential equation or to find roots of a function, however not all schemes will converge for every type of equation.  Similarly, in dynamical systems, we would like to use local properties of an equation to make global assertions, but is this always possible?

This is where the idea of \emph{structural stability} (formally defined in Sec. \ref{Sec: Def}) comes in.  A large enough perturbation of the vector field will change the dynamics of any system.  If a dynamical system is structurally stable it is qualitatively immune to sufficiently small perturbations.  For ``real world'' problems, this tells us that a small measurement error will not change the qualitative behavior of our model.  As an example, we can think of a pendulum, where one of the fixed points corresponds to the pendulum pointing straight down.  If the pendulum is undamped, it will oscillate about this fixed point for some initial position away from the fixed point, however if there is even a small amount of damping, the fixed point becomes an attractor and the pendulum eventually comes to rest.  The added damping changes the vector field only slightly, yet has a significant effect on the behavior of the pendulum.  If our measuring tool cannot detect a small amount of friction and concludes that our pendulum is frictionless, the model would predict a constant amplitude, which observations would belie.

In what follows, we shall endeavor to provide a novel proof of Peixoto's structural stability and density theorems in one-dimension using mainly advanced calculus techniques.  We assume student readers have had a first course on differential equations (or dynamical systems), and are perhaps currently taking advanced calculus/real analysis.  We also encourage interested readers to refer to introductory dynamical systems textbooks (e.g. \cite{Strogatz94, Perko01, Meiss07, BlanchardDevaneyHallODE}) and the seminal work of Smale \cite{Smale1967} while reading this paper.  Before we begin, however, it is worth noting some important limitations of the 1-D version of the theorem.  In one-dimension, a differential dynamical system is limited in the variety of possible flows.  Indeed, concepts such as fixed points and stability persist, while isolated nontrivial periodic orbits, separatrices and strange attractors, for examples, cannot occur.  For advanced undergraduate or graduate students, however, we feel that understanding the 1-D proof presented here shall aid them in understanding the nuances of higher dimensional flows.

The remainder of this manuscript is organized as follows:  in Sec. \ref{Sec: History} we briefly discuss the historical events leading up to Peixoto's theorem.  Section \ref{Sec: Original} touches on the original structural stability theorems, then we list some definitions, which may aid the reader, in Sec. \ref{Sec: Def}.  Section \ref{Sec: Thm} contains the focus of the topic of interest: Peixoto's structural stability theorem on a 1-D, closed, connected, continuously differentiable manifold, which must be a circle (cf. \cite{Milnor97}) that can be represented as the unit interval with identified end points.  Next, we briefly discuss Peixoto's density theorem in Sec. \ref{Sec: Density}, which follows from the structural stability theorem, and leave a few remarks about the two theorems to Sec. \ref{Sec: Remarks}.  We conclude our study in Section \ref{Sec: Conclusion} with some final words on structural stability.

\section{Historical context}\label{Sec: History}

Maur\'{i}cio Peixoto was born a century ago in Fortaleza, northeastern Brazil on April 15\textsuperscript{th} 1921.  In 1943, Peixoto, mainly known for his contributions to Mathematics, graduated with a Civil Engineering degree from the University of Brazil (Universidade do Brasil) \cite{Sotomayor2001}.  Indeed, like many modern applied mathematicians, much of what Peixoto studied was motivated by ``real world'' observations.  The University is also where he met his wife and collaborator, Mar\'{i}lia Chaves Peixoto.  After visiting the University of Chicago for a couple of years, the Peixotos, Leopoldo Nachbin, and others founded the Instituto de Mathem\'{a}tica Pura e Aplicada (IMPA) in 1953 \cite{Sotomayor2001}.  Such is the legacy of this founding that after only 60 years of existence, the IMPA mathematician, Artur Avila, won the fields medal for his contributions to Dynamical Systems Theory.  Maur\'{i}cio's interest in the structural stability of a dynamical system can be traced back to his visit to Princeton University between 1957 and 1958, where he was mentored by one of the pioneers of modern Dynamical Systems Theory, Solomon Lefschetz.  Soon after, Mar\'{i}lia and Maur\'{i}cio began a prolific research program on structural stability.  In 1959, the Peixotos, with Mar\'{i}lia as the lead author, published a paper on structural stability \cite{PeixotoPeixoto1959}, which would later be essential to the proof of Peixoto's theorem.  Sadly, the world lost a trailblazing mathematician with the untimely passing of Mar\'{i}lia Chaves Peixoto in 1961.  Maur\'{i}cio later published what would come to be known as Peixoto's theorem in 1962 \cite{Peixoto62}.

\section{The original theorem}\label{Sec: Original}

Inspired by the pioneering work of Andronov and Pontryagin \cite{Andronov-Pontryagin37} and encouraged by Solomon Lefschetz, the Brazilian engineer and mathematician, Maur\'{i}cio Matos Peixoto, with significant contributions from Mar\'{i}lia Chaves Peixoto, formulated and proved the first global characterization (that is, results on an entire domain rather than in the interval of a fixed point) of structural stability (Def. \ref{Def: Structurally Stable}) and its density properties on smooth surfaces \cite{PeixotoPeixoto1959, Peixoto1959, Peixoto62} in terms that have become synonymous with the modern theory of dynamical systems. Work that also blazed a path for myriad extensions and generalizations. One of the most powerful aspects of Peixoto's theorems is the way it uses local properties to characterize global features of dynamical systems. Both the structural stability and density theorems are combined in Theorem \ref{Thm: Peixoto2D} (see Perko \cite{Perko01}).  It should be noted that this theorem is the original and may include terms unfamiliar to undergraduate students, and these terms will be defined in the next section.

\begin{customthm}{P}[Peixoto's Structural Stability and Density Theorems]\label{Thm: Peixoto2D}
Let $\dot{x} = f(x)$ be a $C^{1}$ (continuously differentiable) dynamical system on a smooth closed surface.  Then the dynamical system is structurally stable if and only if it satisfies the following properties:
\begin{itemize}
\item[(i)]  All recurrent behavior is confined to finitely many hyperbolic fixed points and periodic orbits.
\item[(ii)]  There are no separatrices; that is, orbits connecting saddle points.
\end{itemize}
Moreover, if $M$ is orientable, then the set of structurally stable systems is $C^{1}$ - open and dense in the collection of all $C^1$ - dynamical systems on the surface.
\end{customthm}

As stated in the previous paragraph, this theorem involves mathematical concepts unfamiliar to many advanced undergraduates interested in studying dynamical systems, and the proof is quite long and complicated. Given the importance of the results, both from a theoretical and applied perspective, the much simpler one-dimensional analog treated in what follows is likely to prove useful for understanding Peixoto's theorems and their generalizations, which comprise an essential part of the modern theory of dynamical systems and its applications.

Consider a dynamical system on the circle $\mathbb{S}^{1}$, to which all smooth, closed, connected one-dimensional manifolds (curves) are equivalent. We can represent this as the unit interval on the real line with the end points identified
\begin{equation}
\mathbb{S}^{1}=\mathbb{R}/\mathbb{Z},
\end{equation}
where $\mathbb{R}/\mathbb{Z}$ denotes the real numbers modulo $1$; that is, for $x, y \in \mathbb{R}$, $x \sim y \Leftrightarrow x \equiv y (\mod 1) \Leftrightarrow x-y \in \mathbb{Z}$.
Now let $f:\mathbb{R}\rightarrow\mathbb{R}$ be continuous, then we can define our dynamical system as
\begin{equation}
\dot{x}=f(x)\text{,}\qquad f(x+1)=f(x)\quad\forall\,x\in\mathbb{R}.
\label{Eq: original}%
\end{equation}
That is, $f$ is a continuous periodic function of period one. This can be simplified by restricting to one period, namely the unit interval $0\leq x\leq 1$ such that $f(0)=f(1)$. Then \eqref{Eq: original} becomes
\begin{equation}
\dot{x}=f(x)\text{,}\qquad f(1)=f(0). \label{Eq: f}%
\end{equation}
In this context the function $f$ is called a \emph{vector field}. The vector field $f$ on $\mathbb{S}^{1}$ is of class $C^{k}$ if the function that has $k$ continuous derivatives where each derivative is identified at the end points; that is, $f:[0,1]\rightarrow\mathbb{R}$ such that $f^{(m)}(0)=f^{(m)}(1)$ $\forall\,0\leq m\leq k$. An example of this is the graph in Fig \ref{Fig: orig}.
\begin{figure}[htbp]
\centering
\stackinset{l}{2mm}{b}{11mm}{\textbf{\large (b)}}{\stackinset{l}{10mm}{t}{2mm}{\textbf{\large (a)}}{\includegraphics[width = 0.9\textwidth, valign=c]{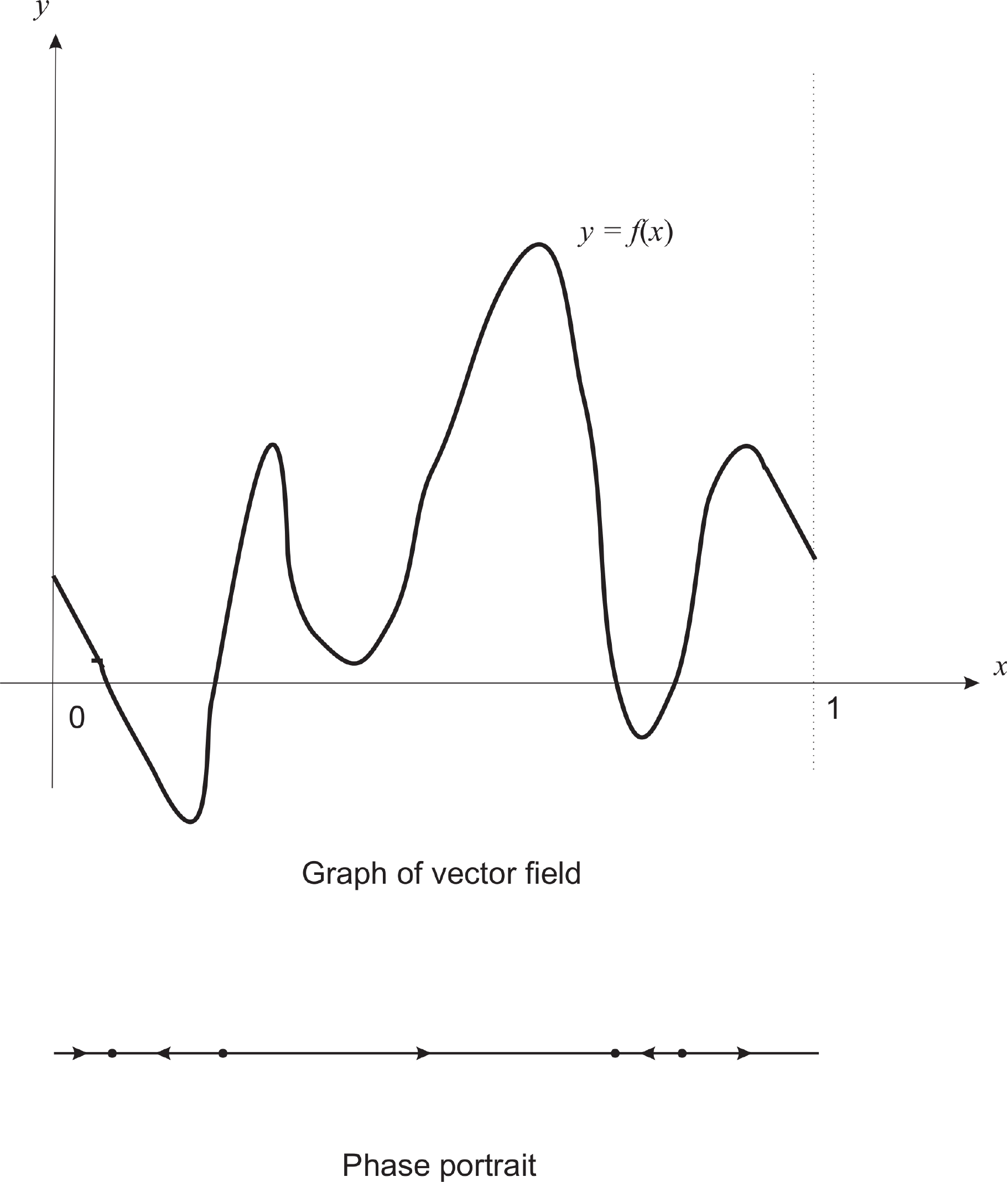}}}%\qquad
%\stackinset{l}{2mm}{t}{2mm}{\textbf{\large (c)}}{\includegraphics[width = 0.4\textwidth, valign=c]{SS_circ.pdf}}
\caption{Example of a vector field of a structurally stable dynamical system on the circle $\mathbb{S}^1$, \textbf{(a)} represented as the unit interval with identified end points, and \textbf{(b)} its phase portrait on the unit interval.}
\label{Fig: orig}
\end{figure}
In addition, applying the vector field to a point $x_0 = x(0)$ for a fixed time $\tau \in \mathbb{R}$ produces the forward iterate $x_1 = x(\tau)$.  Similarly, the backward iterate is $x_{-1} = x(-\tau)$.  Combining these iterates into a single set, $\{x_n : x_n = x(n\tau)\quad \forall n \in \mathbb{Z}\}$, gives us the \emph{orbit} of the system \eqref{Eq: f}.  For example, in Fig. 1(a) if our initial point is between the second and third fixed points, there exists a sequence of points $x_n$ that diverge away from the second fixed point and tend towards the third fixed point.  We say that the orbits of \eqref{Eq: f} are repelled from the second fixed point and attracted to the third fixed point.  In addition to this example, on $\mathbb{S}^1$, there is a variety of possible orbits depending on the orientation of the fixed points.  The simplest type of orbit is a fixed point where $\dot{x}\vert_{x=x_*} = 0 \Rightarrow x(t) = x_*$ for all $t \in \mathbb{R}$.  We may also have periodic orbits on $\mathbb{S}^1$:  suppose the there are no fixed points, then while $\dot{x} \neq 0$, there exists a $\tau \in \mathbb{R}$ such that $x_n = x(n\tau) = x(m\tau) = x_m$ for all $n\in \mathbb{Z}$ and $m\in \mathbb{Z}$ $n \neq m$.

\section{Some key definitions and preliminary theorems}\label{Sec: Def}

Before we begin discussing our main results, let us introduce some definitions that are to play key roles in our analysis. From here on, we shall assume that all of our dynamical systems are at least $C^{1}$ (continuously differentiable).  Also, to simplify the language, we restrict our discussion of manifolds to simple closed $C^{1}$ curves contained in $\mathbb{R}^2$, which are $C^{1}$ equivalent to $\mathbb{S}^{1}$ (cf. \cite{Milnor97}). For those with some knowledge of topology, we note that these curves inherit a metric topology from the standard one on $\mathbb{R}^2$ generated by open balls, and
provide some basic context in the following characterization.

\begin{definition}
\label{Def: topology}
A subset $U$ of $\mathbb{R}^2$ is said to be \emph{open} if for every point $p \in U$ there is an open disk centered at $p$ of positive radius $r$ that is completely contained in $U$, which as trivial examples includes the empty set $\varnothing$ and the whole plane $\mathbb{R}^2$. $(\mathbb{R}^2, \mathcal{U})$, where $\mathcal{U}$ is the collection of all open sets, comprises what is called a \emph{topological space} with topology $\mathcal{U}$, which is referred to as the \emph{Euclidean topology}, \emph{standard topology} or \emph{usual topology} for the plane. Any subset $S$ of the plane inherits the so called \emph{subspace topology} consisting of (open) sets of the form $\{S\cap U:U\in\mathcal{U\}}$. In particular, $\mathbb{R}$ may be considered to be a subspace of $\mathbb{R}^2$, with a topology consisting of sets that are unions of open intervals.

\end{definition}

\begin{definition}
\label{Def: hyperbolic}
A \emph{fixed point}, \emph{stationary point} or \emph{equilibrium point} $x_*$ of \eqref{Eq: f} is one such that $f(x_*)=0$, and is said to be \emph{hyperbolic} if $f'(x_*)\neq0$, otherwise it is said to be \emph{nonhyperbolic}.
\end{definition}

\begin{definition}
\label{Def: homeomorphism}
A map $h : X \rightarrow Y$, where $X$ and $Y$ are subspaces of $\mathbb{R}^2$ with the standard topology, is said to be a \emph{homeomorphism} if it is a bijective (one-to-one and onto) bicontinuous (continuous with continuous inverse) map. If the homeomorphism and its inverse are continuously differentiable, it is called a $C^1$-\emph{diffeomorphism}.
\end{definition}

For example in \cite{Meiss07} it is shown the maps $h:(0,\infty)\rightarrow(0,1)$ defined by $h=1/(1+x^{2})$ and $h:\mathbb{S}^{1}\rightarrow \mathbb{S}^{1}$ defined by $h(x)=x+a\cos x$ for $|a|<1$ are \emph{homeomorphisms}.

\begin{definition}
\label{Def: manifold}
A subset $M$ of $\mathbb{R}^2$ is said to be a 1-dimensional $C^1$-\emph{manifold}, or a 1-dimensional $C^1$-\emph{submanifold} of $\mathbb{R}^2$, if each of its points is contained in an open set that is $C^1$-diffeomorphic to an open set of $\mathbb{R}$ with the usual topology. With this notation, the equivalence of simple closed curves and the unit interval with identified end points $\mathbb{S}^{1}$ can be stated more precisely by saying they are $C^1$-diffeomorphic.
\end{definition}

\begin{definition}
\label{Def: topyequiv}
Two dynamical systems $\dot{x}=f(x)$ and $\dot{y}=g(y)$ on $M$ are \emph{topologically equivalent} if there is a homeomorphism $h$ of $M$ such that $h $ maps oriented (by increasing time) orbits of the first system onto oriented orbits of the second system. Such an $h$ is called a \emph{topological equivalence} between the systems. %This is a weaker form of the equivalence known as a \emph{conjugacy} $h$, for which $h\circ \varphi _{t}=\psi _{t}\circ h$, where $\varphi_{t}$ and $\psi_{t}$  are the solutions (or \emph{flows}) of $f$ and $g$, respectively.
\end{definition}

It is easy to verify that a translation of $-1/4$, which corresponds to $y=h(x)=x-1/4$, is a topological equivalence between the dynamical systems $\dot{x}=\sin(2\pi x)$ and $\dot{y}=\cos(2\pi y)$ on the the unit interval with identified end points $\mathbb{S}^{1}$. A two-dimensional example (in the plane $\mathbb{R}^{2}$) in \cite{Perko01} shows that the linear system $\dot{x}=Ax$ is topologically equivalent to $\dot{y}=By$ where
\begin{equation*}
A=\left[
\begin{array}
[c]{cc}%
-1 & -3\\
-3 & -1
\end{array}
\right]  \quad\text{and}\quad B=\left[
\begin{array}
[c]{cc}%
2 & 0\\
0 & -4
\end{array}
\right]
\end{equation*}
via the homeomorphism
\begin{equation*}
h(x)=\frac{1}{\sqrt{2}}\left[
\begin{array}
[c]{cc}%
1 & -1\\
1 & 1
\end{array}
\right]  x.
\end{equation*}
To see this we note that the origin is the only fixed point of both systems, and for any other initial point $\left(x_{1}^{0}, x_{2}^{0}\right)$, $h$ maps the solution of $\dot{x} = Ax$ beginning at this initial point onto the solution $\left(y_{1}^{0}e^{2t}, y_{2}^{0}e^{-4t}\right)$, with $\left(y_{1}^{0}, y_{2}^{0}\right) = h\left((x_{1}^{0}, x_{2}^{0})\right) = 1/\sqrt{2}\left(x_{1}^{0} - x_{2}^{0}, x_{1}^{0} + x_{2}^{0}\right)$.  This topological equivalence is illustrated in Fig. \ref{Fig: Topological Equivalence}.
\begin{figure}[htb]
\centering
\stackinset{l}{35mm}{t}{2mm}{\textbf{\large (a)}}{\includegraphics[width=.45\textwidth]{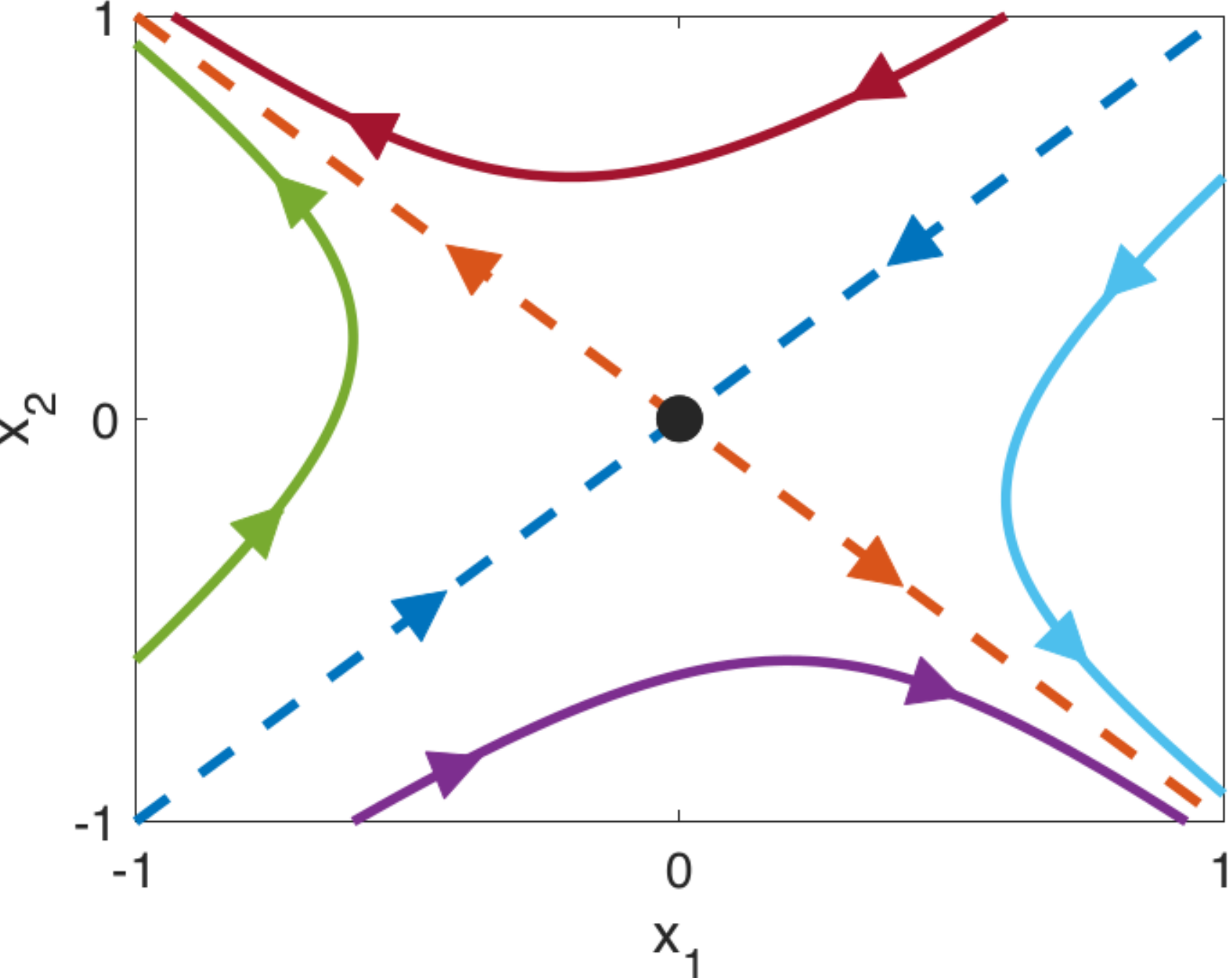}}\qquad
\stackinset{l}{9mm}{t}{2mm}{\textbf{\large (b)}}{\includegraphics[width=.45\textwidth]{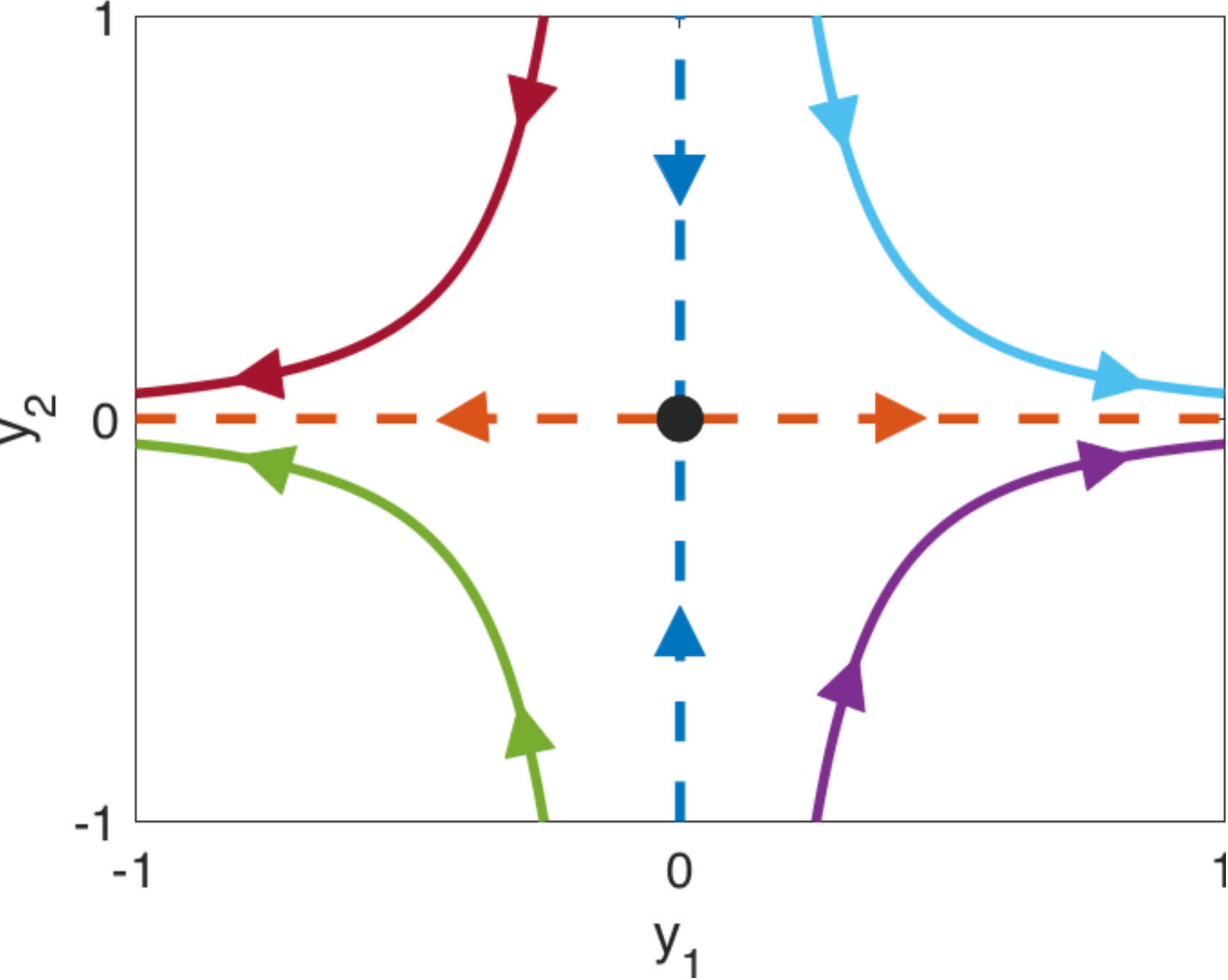}}
\caption{Phase planes of $\dot{x} = Ax$ \textbf{(a)} and $\dot{y} = By$ \textbf{(b)}.  The blue lines with the arrows going into the fixed point correspond to the stable direction of each, and the red/orange lines with arrows going away from the fixed point (black dot) correspond to the unstable direction.  The other colors correspond to representative trajectories where the homeomorphism $h(x)$ maps the trajectories of the same color from \textbf{(a)} to \textbf{(b)}.}
\label{Fig: Topological Equivalence}
\end{figure}

\begin{definition}
Denote the class of $C^{1}$ maps of the unit interval with identified end points by $C^{1}(\mathbb{S}^{1})$. Then
\begin{equation*}
\left\Vert f\right\Vert_1 := \sup\{\left\vert f(x)\right\vert : x\in \mathbb{S}^{1}\}+\sup\{\left\vert f^{\prime}(x)\right\vert :x\in\mathbb{S}^{1}\},
\end{equation*}
defines a norm on $C^{1}(\mathbb{S}^{1})$ called the \emph{$C^{1}$-norm}. This norm generates a topology, called the \emph{$C^{1}$-topology}, in the usual way via the open $\epsilon$-intervals $B_{\epsilon}^{1}(f):= \{g\in C^{1}(\mathbb{S}^{1}) : \left\Vert g-f\right\Vert _{1} < \epsilon\}$.
\end{definition}

\begin{definition}
\label{Def: Structurally Stable}
The dynamical system $\dot{x}=f(x)$ with $f\in C^{1}(\mathbb{S}^{1})$ is said to be $C^{1}$ \emph{structurally stable} if for every $\epsilon>0$ sufficiently small and any $g\in B_{\epsilon}^{1}(f)$, the systems $\dot{x} = f(x)$ and $\dot{y} = g(y)$ are topologically equivalent.
\end{definition}

We note that it is not difficult to imagine how these last two definitions can be generalized to any finite-dimensional closed (compact and without boundary) manifolds.

It is useful to take note of the following rather simple characterization of topological equivalence for a pair of $C^{1}$ dynamical systems, (i) $\dot{x}=f(x)$ and (ii) $\dot{y}=g(y)$ on the unit interval with identified end points $\mathbb{S}^{1}$.
\begin{customthm}{TE}\label{Thm: TE}
A homeomorphism $h:\mathbb{S}^{1}\rightarrow \mathbb{S}^{1}$ is a topological equivalence from (i) to (ii) if and only if
\begin{subequations}
\begin{align}
&h\left(\{x\in \mathbb{S}^1: f(x)=0\}\right) := h(f^{-1}(0)) = \{y\in \mathbb{S}^1: g(y)=0\} := g^{-1}(0)\label{Eq: roots}\\
&\text{and either}\nonumber\\
&(-f)^{-1}(0,\infty) = (g\circ h)^{-1}(0,\infty) \text{  and  } (-f)^{-1}(-\infty,0) = (g\circ h)^{-1}(-\infty,0)\label{Eq: sgn1}\\
&\text{  if $h$ is a decreasing function,}\nonumber\\
&\text{     or}\nonumber\\
&f^{-1}(0,\infty) = (g\circ h)^{-1}(0,\infty) \text{  and  } f^{-1}(-\infty,0) = (g\circ h)^{-1}(-\infty,0) \label{Eq: sgn2}\\
&\text{  if $h$ is an increasing function}\nonumber
\end{align}
\label{Eq: TE}
\end{subequations}
\end{customthm}

\begin{proof}
For sufficiency, let us assume that $h$ is a topological equivalence.  The simplest orbits to choose from (i) are the fixed points, where $f(x) = 0$.  Since $h$ maps orbits of (i) to (ii), it must map the fixed point (zero) set of $f$ homeomorphically onto the fixed point (zero) set of $g$.  Furthermore, since $h$ maps oriented orbits of (i) to (ii), the signs of $f(x)$ and $g\circ h(x)$ between corresponding fixed points must be consistent with the behavior (increasing or decreasing) of $h(x)$; that is, if $h$ is decreasing, $f(x)$ and $g\circ h(x)$ will have opposite signs, and if $h$ is increasing, $f(x)$ and $g\circ h(x)$ will have the same sign.

For necessity let us assume \eqref{Eq: TE} holds.  Then by \eqref{Eq: roots} $h$ maps the fixed point set of (i) homeomorphically onto the fixed point set of (ii).  Moreover, \eqref{Eq: sgn1} and \eqref{Eq: sgn2} imply that $h$ maps positively time-oriented orbits of (i) onto positively time-oriented orbits of (ii).  Consequently $h$ is a topological equivalence.
\end{proof}

Sometimes it may be difficult to see why $h$ is increasing or decreasing on $\mathbb{S}^{1}$.  If we take the example $y = h(x) = -x + 1/2$ where $x, y \in \mathbb{S}^{1}$ we get that $h(0) = 1/2 \Rightarrow y = 1/2$ and $h(3/4) = -1/4 \Rightarrow y = 3/4$.  However, notice that $y = -1/4$ and $y = 3/4$ will correspond to the same point, just like in a circle where $-\pi/2$ and $3\pi/2$ correspond to the same point.  Therefore, $h$ is indeed decreasing even if $x = 0$ maps to $y = 1/2$ and $x = 3/4$ maps to $y = 3/4$.  For other characterizations of topological equivalence similar to Theorem \ref{Thm: TE} see \cite{BlanchardDevaneyHallODE}.

\subsection{Bump functions}\label{Sec: bump}

Throughout the sufficiency portion of the proof of the one-dimensional analog of Theorem \ref{Thm: Peixoto2D}, we will make extensive use of bump functions. A simple example of a bump function is $\psi:\mathbb{R}\rightarrow\mathbb{R}$ defined as \eqref{Eq: simple bump} and plotted in Fig. \ref{Fig: bump}.
\begin{equation}
\psi(x):=%
\begin{cases}
e^{\frac{-1}{1-x^{2}}} & \text{for}\;x\in(-1,1),\\
0 & \text{for}\;x\notin(-1,1);
\end{cases}
\label{Eq: simple bump}%
\end{equation}
It should be noted that we reserve the variable $\psi$ specifically for bump functions, and use $\eta$ to speak of perturbations in general terms.

\begin{figure}[htb]
\centering
\includegraphics[width=.9\textwidth]{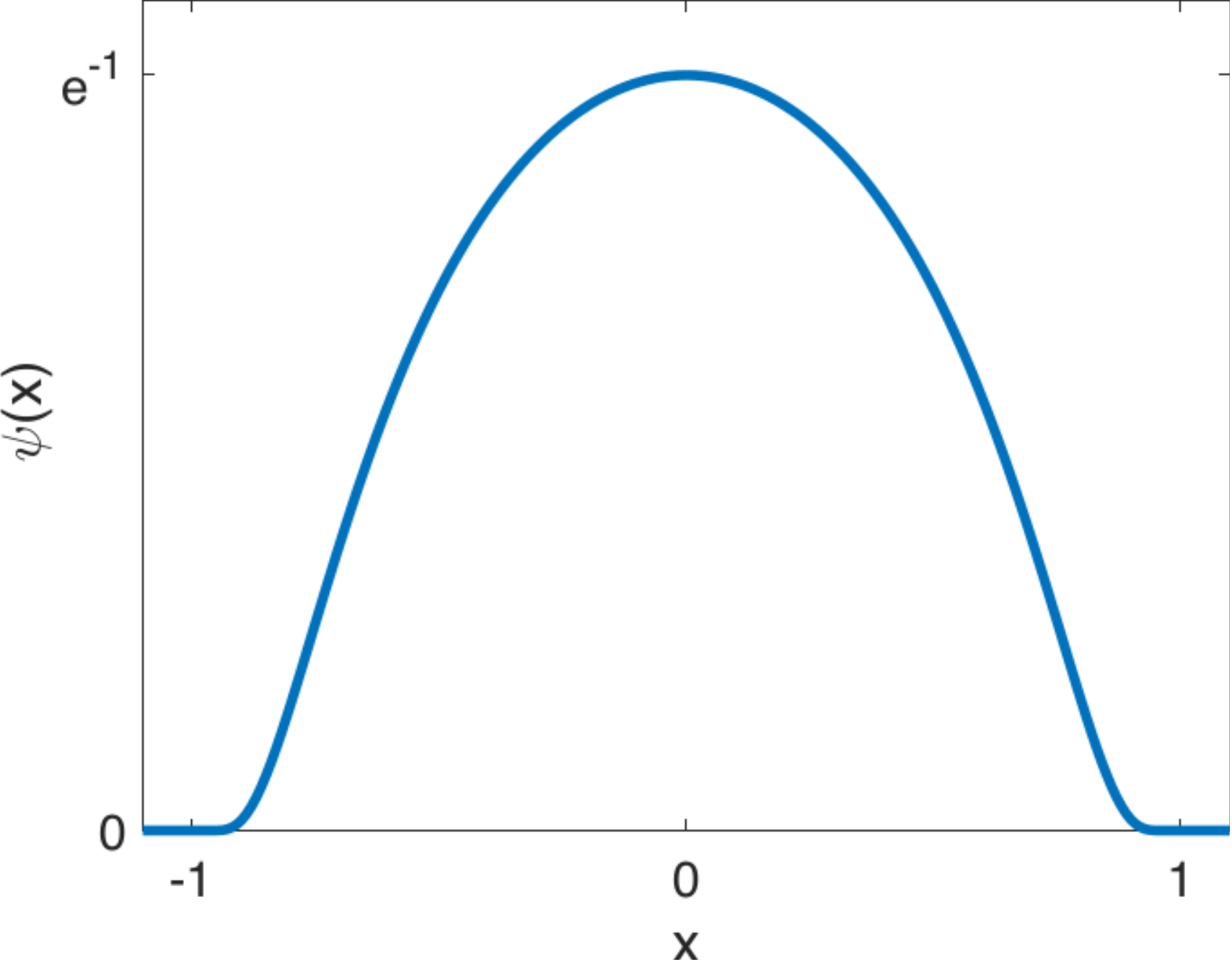}\caption{Example of a bump function.}%
\label{Fig: bump}%
\end{figure}

We can restrict the bump to any interval, which for our proof shall correspond to an interval of our choosing. Let $x_{0}$ be the center of the interval, and $r$ $(>0)$ be the radius. Then, denoting the $r$-interval as $B_{r}\left(  x_{0}\right)  :=\{x\in\mathbb{R}:$ $\left\vert x-x_{0}\right\vert <r\}$, our bump function becomes
\begin{equation}
\psi(x):=%
\begin{cases}
\exp\left(  \frac{-r^{2}}{r^{2}-(x-x_{0})^{2}}\right)  & \text{for}\;x\in
B_{r}(x_{0}),\\
0 & \text{for}\;x\notin B_{r}(x_{0});
\end{cases}
\end{equation}
This can be specified to any interval $[a,b]$, which corresponds to a interval centered at $(a+b)/2$ with a radius of $(b-a)/2$: namely,
\begin{equation}
\psi(x):=
\begin{cases}
\exp\left(  -\left(  \frac{b-a}{2}\right)  ^{2}\middle/\left[  \left(\frac{b-a}{2}\right)  ^{2}-\left(  x-\frac{a+b}{2}\right)  ^{2}\right]\right)  & \text{for}\;x\in(a,b),\\
0 & \text{for}\;x\notin(a,b);
\end{cases}
\end{equation}
We also note that it can be easily verified using basic calculus that the bump function \eqref{Eq: simple bump} satisfies
\begin{equation}
\left\Vert\psi\right\Vert_1 < e^{-1}\left[1+6e^{-1}(b-a)^{-1}\right],
\label{Eq: bounded}
\end{equation}
so once the interval has been specified, for every $\epsilon > 0$ there exists an arbitrary constant $\sigma > 0$ such that $\left\Vert \sigma\psi\right\Vert _{1} < \epsilon$. This will be an important fact to keep in mind for the sufficiency portion of our proof. In many instances we will use a scaled bump function to obtain small $C^{1}$ perturbations of given functions in $C^{1}(\mathbb{S}^{1})$.

\section{Peixoto's theorem on $\mathbb{S}^1$}\label{Sec: Thm}

We note once again that every closed and connected one-dimensional $C^1$ manifold is diffeomorphic to a unit interval with identified end points, as shown in \cite{Milnor97}. So, it suffices to restrict our attention to the unit interval with identified end points (or more accurately a $1$-sphere) $\mathbb{S}^{1}$ in the one-dimensional analog of Theorem \ref{Thm: Peixoto2D} that follows.

\setcounter{thm}{0}
\begin{thm}
\label{Thm: peixoto}
Suppose \eqref{Eq: f} is a $C^{1}$ dynamical system on $\mathbb{S}^{1}$. Then \eqref{Eq: f} is $C^{1}$ structurally stable if and only if it has finitely many fixed points, all of which are hyperbolic.
\end{thm}

\begin{proof}
For necessity, let us first prove the result for a dynamical system with no fixed points. Suppose \eqref{Eq: f} has no fixed points, then $|f(x)|>0$ on $[0,1]$. Without loss of generality, assume $f$ is positive, which means the phase space consists of a single periodic counterclockwise orbit. Since $f$ is continuous, there is an $\epsilon_{0}>0$ such that $f(x) > \epsilon_{0}\; \forall\, x \in [0,1]$. Consider the dynamical system
\begin{equation}
\dot{y}=g(y), \label{Eq: g}
\end{equation}
where $g\in C^{1}(\mathbb{S}^{1})$ is any function in an $\epsilon$-neighborhood of $f$ in the $C^1$-topology with
\begin{equation}
\left\Vert g-f\right\Vert _{1}<\epsilon<\epsilon_{0}, \label{Eq: C1}%
\end{equation}
then $g$ must also be positive on $\mathbb{S}^{1}$. Therefore, it follows from \eqref{Eq: TE} that the identity map is a topological equivalence, so \eqref{Eq: f} is structurally stable.

As a demonstration consider the function $f(x) = \sin(4\pi x) + 4$ on the unit interval with identified end points.  Then for any $\epsilon$ sufficiently small, if a function $g(y)$ satisfies \eqref{Eq: C1}, it is not possible to introduce any new fixed points or to change the direction of the flow; that is, the dynamics of $\dot{x} = f(x)$ and $\dot{y} = g(y)$ remains the same.  This is illustrated in Fig.\ref{Fig: NoFP}.
\begin{figure}[htbp]
\centering
\includegraphics[width = 0.9\textwidth]{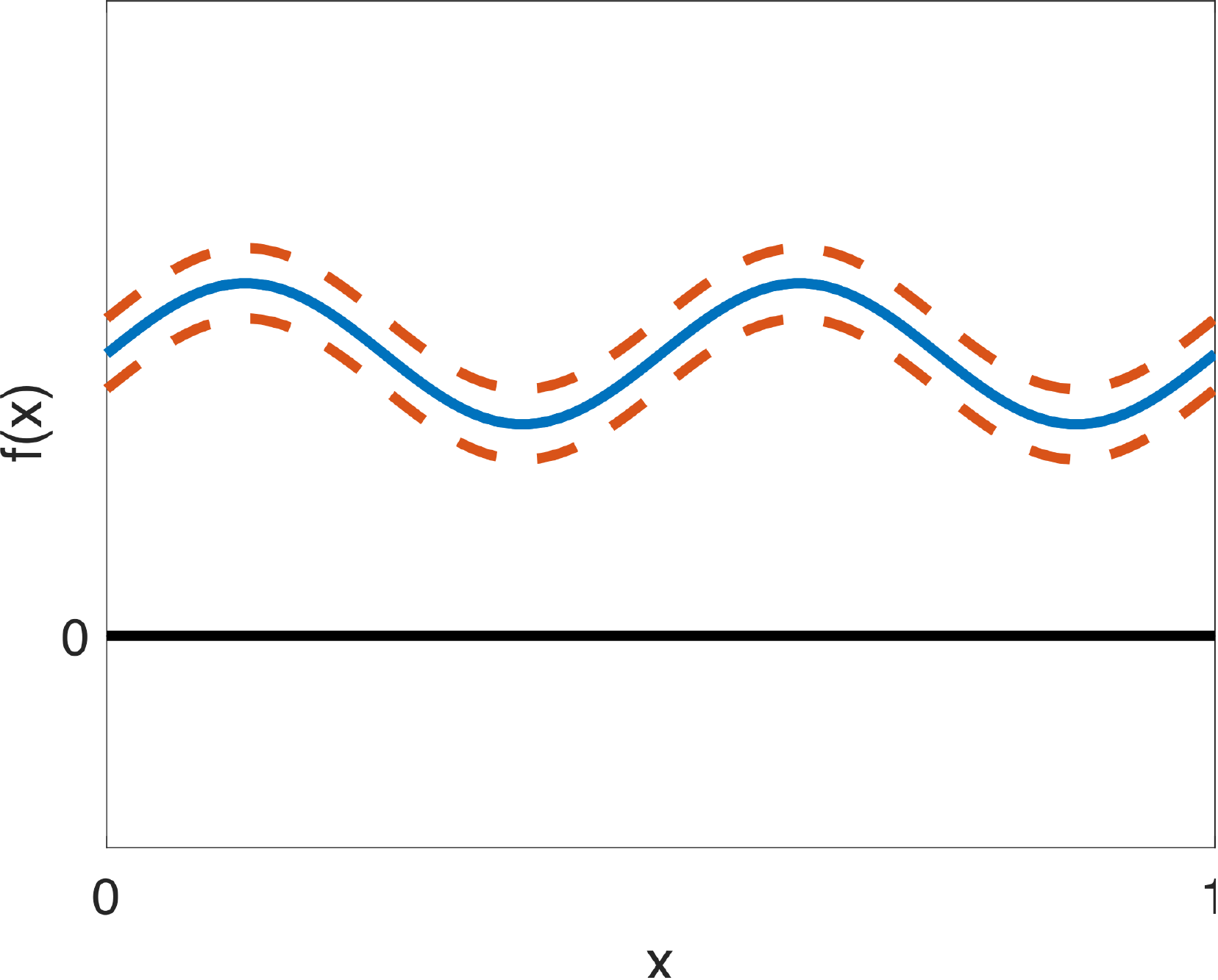}
\caption{Example of a function (solid blue curve), $f(x)$, with no fixed points.  While, in general, we do have to take the derivative into account, as a demonstration we show how any other function $g(y)$ will not cross over the horizontal axis as long as it remains within the bounds (dashed red curves), which represent a simplified version of \eqref{Eq: C1}.}\label{Fig: NoFP}
\end{figure}

Next we prove the result for finitely many hyperbolic fixed points. Suppose $f(x_{k}) = 0$ and $f^{\prime}(x_{k}) \neq 0$ for $k=1,2,\ldots, m$. We may assume none of them is an endpoint of the unit interval, since this can always be accomplished via a simple translation of the period interval. Let us order them as follows, corresponding to a counterclockwise ordering on the unit interval with identified end points:
\begin{equation*}
x_{1}<x_{2}<\cdots<x_{m}.
\end{equation*}
We note that there is an arc of $\mathbb{S}^{1}$ corresponding to each of the intervals $[x_{1},x_{2}],\ldots,[x_{m-1},x_{m}]$ and also an arc corresponding to $[x_{m},1]\cup[0,x_{1}]$, which we denote as the interval $[x_{m},x_{1}]$.

Observe that between any two fixed points in the above intervals, $f$ does not change sign, and in some $\delta$-interval of every fixed point $f^{\prime}$ is nonzero owing to the hyperbolicity and continuity. So, we can select $\epsilon_{0}$ and $\delta > 0$ small enough such that $|f^{\prime}(x)| \geq 2\epsilon_{0}$ on $[x_{k}-\delta, x_{k}+\delta]$ for $k=1,2,\ldots,m$, where the intervals $[x_{k}-\delta, x_{k}+\delta]$ are disjoint. Hence, $f$ is monotonic and has a single zero on each of these intervals. Now define $K(\delta)$ as the closure of the complement of these intervals, which is just a disjoint union of closed intervals itself. Since $f$ does not change sign between any two sequential fixed points, $f$ is nonzero in $K(\delta)$. Furthermore, it follows from continuity of $f$ that there is $0 < \epsilon_1 < \epsilon_0$ such that $|f(x)| \geq 2\epsilon_{1}\; \forall\, x \in K(\delta)$.

Now we show that structural stability is satisfied; in particular, for any $g\in C^{1}\left(\mathbb{S}^{1}\right)$ satisfying
\begin{equation}
\left\Vert f-g\right\Vert_1 < \epsilon < \min\{\epsilon_0, \epsilon_1\}
\label{Eq: pert}
\end{equation}
the dynamical system \eqref{Eq: g} is topologically equivalent to \eqref{Eq: f}. That is, there is a homeomorphism $h: \mathbb{S}^1 \rightarrow \mathbb{S}^1$ mapping oriented orbits of \eqref{Eq: g} to  \eqref{Eq: f}. By \eqref{Eq: pert} $g$ has exactly one zero, denoted as $y_{k}$, on each interval $[x_k-\delta, x_k+\delta]$ and no zeros on $K(\delta)$, and furthermore $g^{\prime}(y_{k})\neq0$ and $g$ has the same sign on $(y_k, y_{k+1})$ as $f$ does on $(x_k, x_{k+1})$ whenever $1 \leq k \leq m-1$ and on $(y_{m},y_{1})$ and$(x_{m},x_{1})$. Consequently, $f$ and $g$ have the same number of fixed points, $\epsilon$-close to one another and have the same signs on the corresponding intervals between the respective zeros.

We select $h$ to be a piecewise linear homeomorphism with the following properties: $h(0)=h(1)=0$ and also $h(x_{k})=y_{k}$ for $k=1,2,\ldots,m$, and linear on the intervals $(0,x_1),\, (x_1,x_2),\, \ldots,\, (x_{m-1},x_m),(x_m,1)$; namely
\begin{equation}
h(x):=\left\{
\begin{array}
[c]{cc}%
\left(  \frac{y_{1}}{x_{1}}\right)  x, & 0\leq x\leq x_{1}\\
\frac{\left(  y_{k+1}-y_{k}\right)  x+\left(  y_{k}x_{k+1}-x_{k}%
y_{k+1}\right)  }{\left(  x_{k+1}-x_{k}\right)  }, & x_{k}\leq x\leq
x_{k+1},1\leq k\leq m-1\\
\left(  \frac{y_{m}}{x_{m}-1}\right)  (x-1), & x_{m}\leq x\leq1
\end{array}
\right.
\end{equation}
This proves the necessity of the hypothesis owing to \eqref{Eq: TE}.

Again, as a demonstration consider the function $f(x) = \sin(4\pi x) + \cos(2\pi x)$ on the unit interval.  Then for any $\epsilon$ sufficiently small, if a function $g(y)$ satisfies \eqref{Eq: pert}, it is neither possible to change the number of fixed points nor the direction of the flow; that is, the dynamics of $\dot{x} = f(x)$ and $\dot{y} = g(y)$ remains the same.  It should be noted that the derivative plays a very important role in this case since a large enough change can convert any of the hyperbolic fixed points into a nonhyperbolic fixed point.  This case is illustrated in Fig. \ref{Fig: FiniteFP}.
\begin{figure}[htbp]
\centering
\includegraphics[width = 0.9\textwidth]{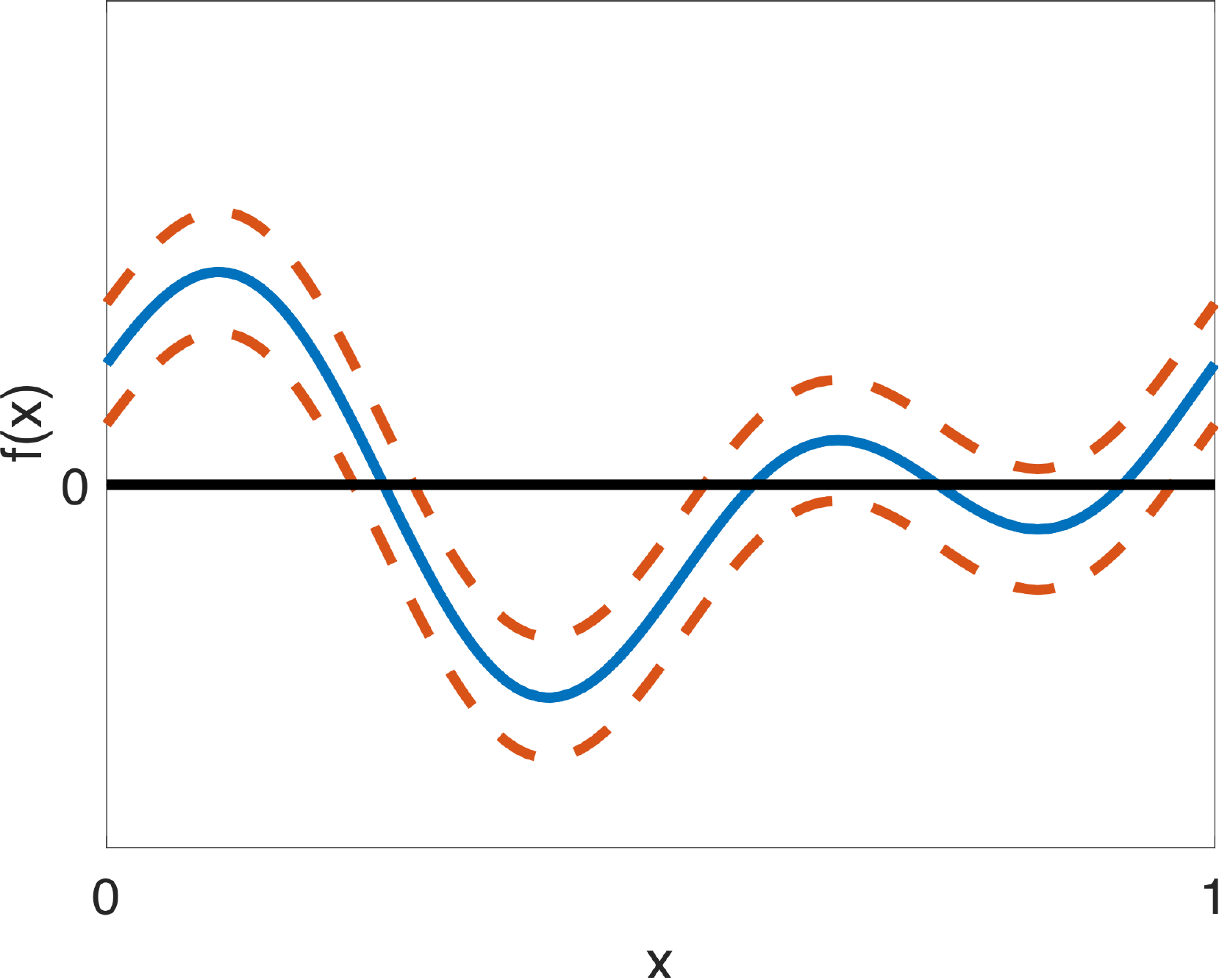}
\caption{Example of a function (solid blue curve), $f(x)$, with finitely many hyperbolic fixed points.  While we have to be especially careful of the derivative in this case, as a demonstration we show how a reasonable function, $g(y)$, satisfying the conditions of \eqref{Eq: pert} will remain within the bounds (dashed red curves), which represent a simplified version of \eqref{Eq: pert}.}\label{Fig: FiniteFP}
\end{figure}

\bigskip

For sufficiency, we show that if the fixed point hypothesis is not satisfied, then \eqref{Eq: f} is not structurally stable. To accomplish this, we need to analyze all the cases in which the hypothesis may be violated. In each case we show that if we add an arbitrarily small $C^{1}$-perturbation $\eta$ of the right form to the original system to obtain $g(x)=f(x)+\eta(x)$, we create a system that is not topologically equivalent to \eqref{Eq: f}. The analysis is focused on nonhyperbolic fixed points, with the intention of showing that an arbitrarily small perturbation can change the homeomorphism type of the original fixed point set, thus ensuring, owing to \eqref{Eq: TE}, that the perturbed system cannot be topologically equivalent to the original.\\

\noindent \textbf{Case 1}: Suppose that $x_*$ is an isolated nonhyperbolic fixed point (so that $f^{\prime}(x_{\ast})=0$) across which $f$ does not change sign. It can be assumed, without loss of generality, that $f(x)>0$ in some punctured $\delta$-interval $\mathring{B}_{\delta}(x_*):=B_{\delta}(x_*) \smallsetminus \{x_*\}$ of $x_*$ as shown in Fig \ref{Fig: case1}.
\begin{figure}[htb]
\centering
\includegraphics[width=.9\textwidth]{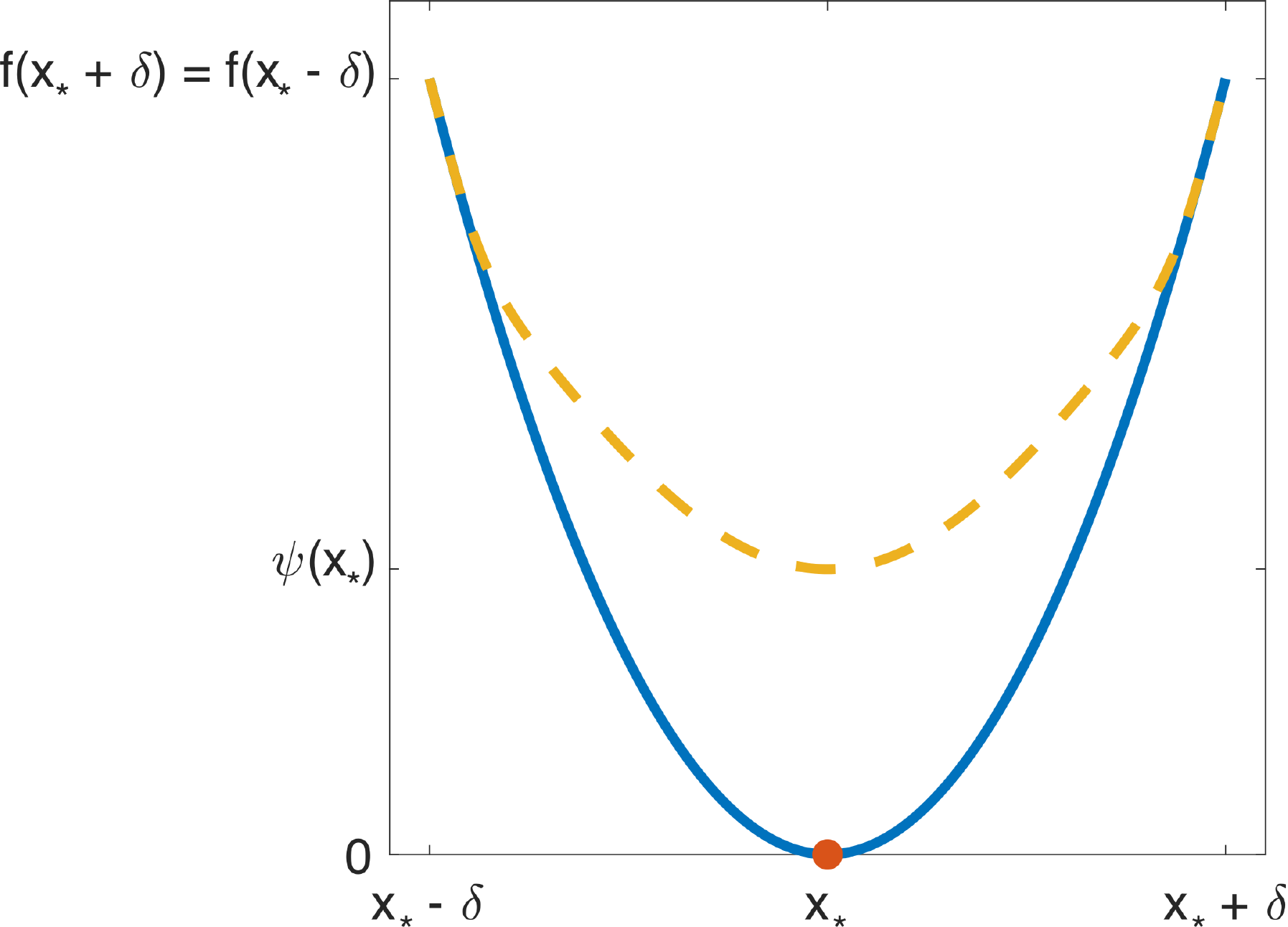}\caption{An example \eqref{Eq: case1Func} of a function (bottom solid blue curve), $f$, that does not change sign in a small interval, $(x_* - \delta, x_* + \delta)$, centered at a nonhyperbolic fixed point, $x_*$.  The function is being perturbed by a bump function restricted to the small interval, and the perturbed function (top dashed yellow curve) is shown to have lost the fixed point, $x_*$.}
\label{Fig: case1}
\end{figure}
Note that if $x_*$ is the only fixed point, the perturbation $g(x)=f(x)+\epsilon$ defines a dynamical system with an empty fixed point set, denoted as $g^{-1}(0)\in \varnothing$, while $f^{-1}(0)\in \{x_*\}$, so it follows from (TE) that $\dot{x}=f$ and $\dot{y}=g$ are topologically inequivalent.

If $x_{\ast}$ is not unique, we have to localize the above analysis.  For the purpose of demonstration, it may be useful for the reader to think of an example function
\begin{equation}
f(x) = \begin{cases}
f_\text{in}\\
f_\text{out}
\end{cases} = \begin{cases}
(x-x_*)^2 & \text{for $x\in (x_* - \delta, x_* + \delta)$},\\
f_\text{out} & \text{for $x\in \mathbb{S}^1\smallsetminus (x_* - \delta, x_* + \delta)$};
\end{cases}
\label{Eq: case1Func}
\end{equation}
where $f_\text{out}$ is an arbitrary function with properties satisfying the hypothesis of the theorem.  We would like to perturb the system in such a way as to annihilate the fixed point without affecting the function outside of this interval.  By perturbing the function only within the $\delta$ interval, we avoid translating or adding and removing the same number of the same type of fixed points, which would defeat the purpose of the perturbation.  This can be accomplished by using a bump function (\textit{cf}. section \ref{Sec: bump}).  Since $x_{\ast}$ is the only fixed point, which we may assume is not an end point of the unit interval, in $(x_{\ast}-\delta,x_{\ast}+\delta)$, we define
\begin{equation}
\psi(x):=
\begin{cases}
\exp\left(  \frac{-\delta^{2}}{\delta^{2}-(x-x_{\ast})^{2}}\right)  &
\text{for}\;x\in(x_{\ast}-\delta,x_{\ast}+\delta),\\
0 & \text{for}\;x\notin(x_{\ast}-\delta,x_{\ast}+\delta);
\end{cases}
\label{Eq: case1}
\end{equation}
It follows from \eqref{Eq: bounded} that for any $\epsilon>0$, there exists a $\sigma>0$ such that $\left\Vert\sigma\psi\right\Vert_1 < \epsilon$. Hence, if $g=f+\sigma\psi$, $\left\Vert f-g\right\Vert _1 < \epsilon$ and $\dot{y}=g$ has no fixed point in $(x_{\ast}-\delta,x_{\ast}+\delta)$. Therefore, $f^{-1}(0)$ and $g^{-1}(0)$ cannot be homeomorphic, which means that $f$ is not structurally stable.\\

\noindent \textbf{Case 2}: Let $x_*$ be a fixed point across which $f$ changes sign, but $f'(x_*)\neq 0$,which can be assumed, without loss of generality, to be as shown in Fig. \ref{Fig: case2}. Again, we define $\delta>0$ such that $f$ is not zero in $\mathring{B}_{\delta}(x_*)$.
\begin{figure}[htb]
\centering
\stackinset{l}{13mm}{t}{2mm}{\textbf{\large (a)}}{\includegraphics[width=.45\textwidth]{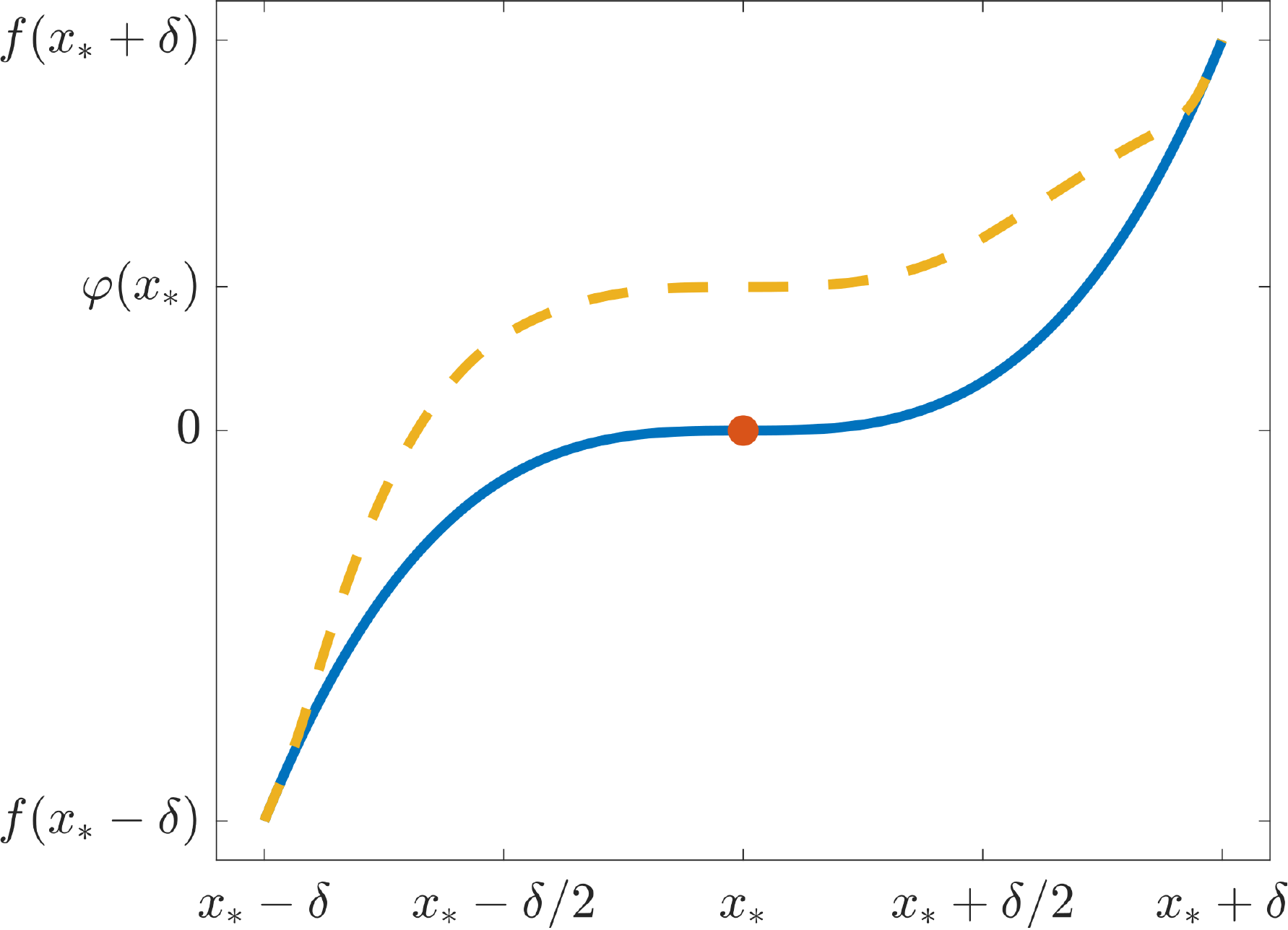}}\qquad
\stackinset{l}{13mm}{t}{2mm}{\textbf{\large (b)}}{\includegraphics[width=.45\textwidth]{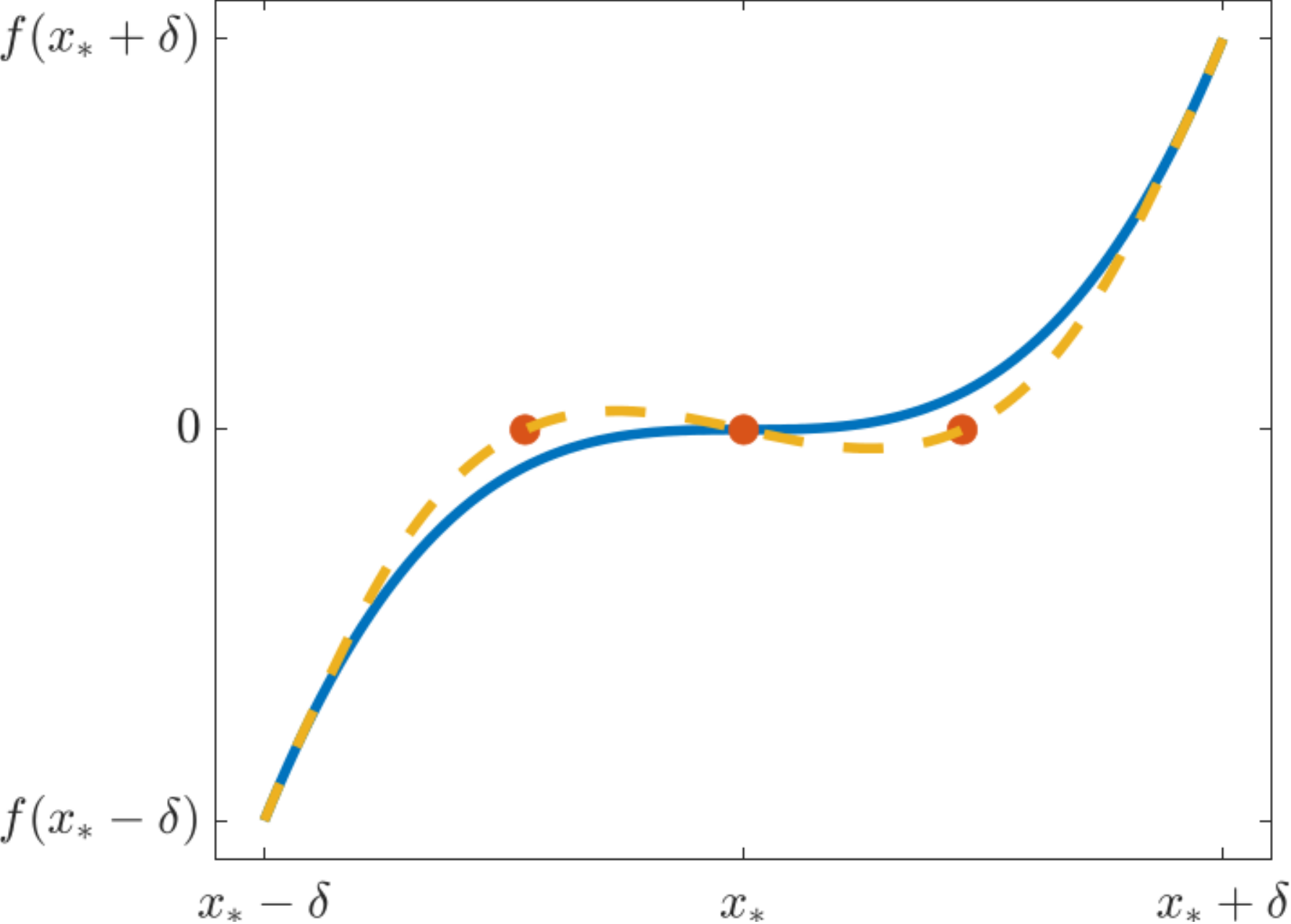}}
\caption{An example \eqref{Eq: case2Func} of a function (solid blue curve), $f$, that changes sign in a small interval, $(x_* - \delta, x_* + \delta)$, centered at a nonhyperbolic fixed point, $x_*$.  \textbf{(a)} The function is perturbed by an extended bump function \eqref{Eq: case2} restricted to the small interval, and the perturbed function (dashed yellow curve) is shown to have converted the fixed point, $x_*$, from nonhyperbolic to hyperbolic.  \textbf{(b)} The function is perturbed by a ``double'' bump function \eqref{Eq: case2v} restricted to the small interval, and the perturbed function (dashed yellow curve) is shown to have gained two hyperbolic fixed points, and converted the original nonhyperbolic fixed point, $x_*$, into yet another hyperbolic fixed point.}
\label{Fig: case2}
\end{figure}
For this case, we shall show there is an arbitrarily small $C^{1}$-perturbation confined to $B_{\delta}(x_{\ast})$, which has three hyperbolic fixed points in this interval instead of one nonhyperbolic one.  Again, for the purpose of demonstration, it may be useful for the reader to think of an example function
\begin{equation}
f(x) = \begin{cases}
f_\text{in}\\
f_\text{out}
\end{cases} = \begin{cases}
(x-x_*)^3 & \text{for $x\in (x_* - \delta, x_* + \delta)$},\\
f_\text{out} & \text{for $x\in \mathbb{S}^1\setminus (x_* - \delta, x_* + \delta)$};
\end{cases}
\label{Eq: case2Func}
\end{equation}
where $f_\text{out}$ is an arbitrary function with properties satisfying the hypothesis of the theorem.

First, we create an extended bump function that is $e^{-1}$ in the closed interval $[x_* - \delta/2, x_* + \delta/2]$ and vanishes in the complement of $(x_* - \delta, x_* + \delta)$; namely,
\begin{equation}
\varphi(x):=\left\lbrace
\begin{array}
[c]{cc}%
\exp\left(  -\left(  \frac{\delta}{2}\right)  ^{2}\middle/\left[  \left(\frac{\delta}{2}\right)  ^{2}-\left(  x-x_{\ast}+\frac{\delta}{2}\right)^{2}\right]  \right)  , & x_{\ast}-\delta<x\leq x_{\ast}-\delta/2\\
e^{-1}, & x_{\ast}-\delta/2\leq x\leq x_{\ast}+\delta/2\\
\exp\left(  -\left(  \frac{\delta}{2}\right)  ^{2}\middle/\left[  \left(\frac{\delta}{2}\right)  ^{2}-\left(  x_{\ast}-x+\frac{\delta}{2}\right)^{2}\right]  \right)  , & x_{\ast}+\delta/2\leq x<x_{\ast}+\delta\\
0, & x\notin(x_{\ast}-\delta,x_{\ast}+\delta)
\end{array}\right\rbrace .
\label{Eq: case2}
\end{equation}
Observe that this function is $C^{\infty}$ on the whole real line except at the points $x_*-\delta/2$ and $x_{\ast}+\delta/2$ where it is only $C^{1}$.  Next, we define a ``double'' bump function by using $\psi(x)$ from \eqref{Eq: case1}.
\begin{equation}
\vartheta(x) := -\frac{2(  x-x_{\ast})\psi(x)}{\delta e},
\label{Eq: case2v}
\end{equation}
which is intended to add fixed points.  Examples of \eqref{Eq: case2} and \eqref{Eq: case2v} are given in Fig \ref{Fig: case2a}.
\begin{figure}[htb]
\centering
\stackinset{l}{12mm}{t}{2mm}{\textbf{\large (a)}}{\includegraphics[width = .46\textwidth]{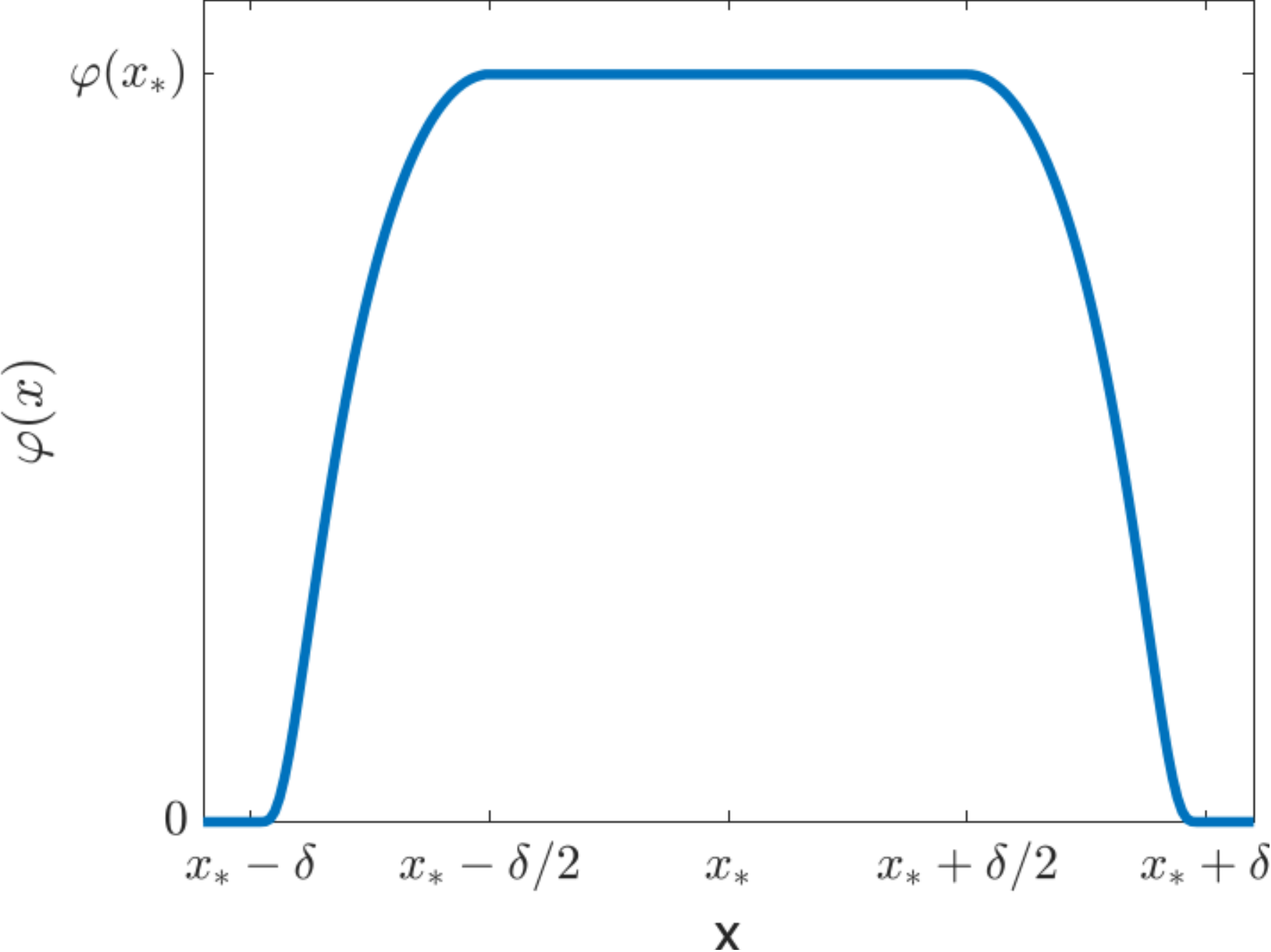}}\qquad
\stackinset{r}{2mm}{t}{2mm}{\textbf{\large (b)}}{\includegraphics[width = .44\textwidth]{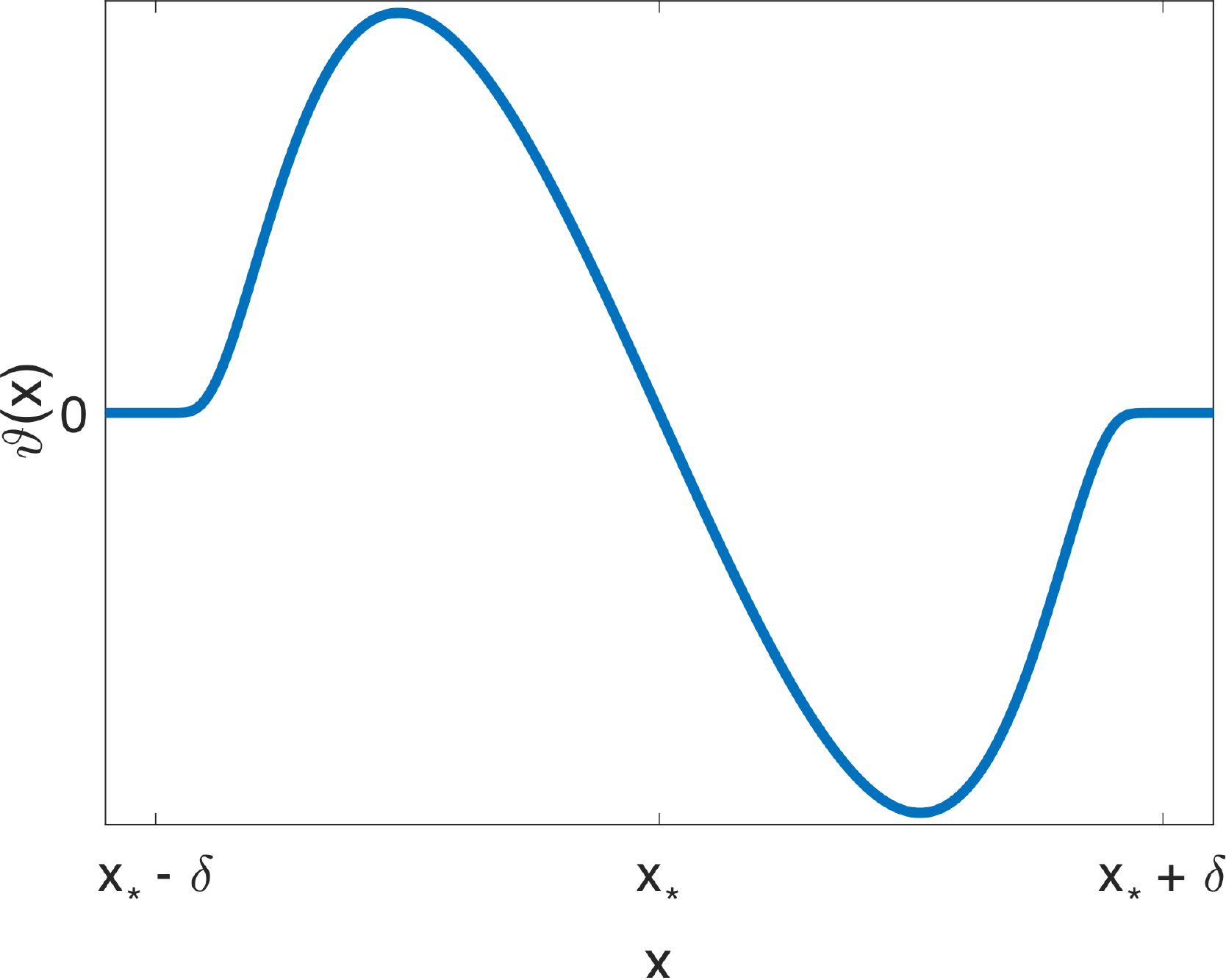}}
\caption{Example plots of \textbf{(a)} \eqref{Eq: case2} and \textbf{(b)} \eqref{Eq: case2v}.}
\label{Fig: case2a}
\end{figure}

Note, as usual, for each $\epsilon>0$ there exists a $\sigma>0$ such that $\left\Vert\sigma\vartheta\right\Vert_1 < \epsilon$.  Moreover, $g:=f+\sigma\vartheta$ has a zero at $x_*$ with $g'(x_*)<0$ and just two other zeros in $B_{\delta}(x_{\ast})$ at points $x_{\ast}\pm\nu$, with $0<\nu<\delta$, and $g'(x_*-\nu) = g'(x_*+\nu) > 0$. As $\left\Vert f-g\right\Vert_1 < \epsilon$ and $f^{-1}(0)$ and $g^{-1}(0)$ are not homeomorphic, \eqref{Eq: TE} implies that $f$ is not structurally stable.\\

\noindent \textbf{Case 3}: An interval of (nonhyperbolic) fixed points, that is, $f(x)=0$ on some interval $[a,b] \subseteq [0,1]$. If $[a,b]=[0,1]$, the addition of an arbitrarily small positive constant changes the fixed point set from all of $\mathbb{S}^{1}$ to the empty set, which proves that such a system cannot be $C^{1}$-structurally stable. On the other hand, if the interval is a proper subset of the unit interval, consider the case where $[a,b]\subseteq (0,1)$ and is isolated from any other points in the fixed point set of $\dot{x}=f$. Accordingly, there is a positive $\delta$ such that $[a-\delta,b+\delta]\subset (0,1)$ and $[a-\delta,b+\delta] \cap f^{-1}(0) = [a,b]$. By analogy with Case 1 and Case 2 above, we consider two subcases: (i) $f$ has the same sign in $(a-\delta,a)$ and $(b,b+\delta)$; and (ii) $f$ has opposite signs in $(a-\delta,a)$ and $(b,b+\delta)$. Naturally, we may assume without loss of generality the sign in (i) is positive, and in (ii) it goes from negative to positive. It is convenient to use the following analog of the bump function, $\psi$, for both (i) and (ii):
\begin{equation}
\hat{\varphi}(x):=\left\{
\begin{array}
[c]{cc}%
\exp\left(  -\left(  \frac{\delta}{2}\right)  ^{2}\middle/\left[  \left(\frac{\delta}{2}\right)  ^{2}-\left(  x-a+\frac{\delta}{2}\right)^{2}\right]  \right)  , & a-\delta<x\leq a-\delta/2\\
e^{-1}, & a-\delta/2\leq x\leq b+\delta/2\\
\exp\left(  -\left(  \frac{\delta}{2}\right)  ^{2}\middle/\left[  \left(\frac{\delta}{2}\right)^{2}-\left(  b-x+\frac{\delta}{2}\right)^{2}\right]  \right), & b+\delta/2\leq x<b+\delta\\
0, & x\notin(a-\delta,b+\delta)
\end{array}
\right.  . \label{Eq: case3}%
\end{equation}
As for any $\epsilon>0$ there is a $\sigma>0$ such that $\left\Vert\sigma\hat{\varphi}\right\Vert_1 <\epsilon$, in subcase (i) the perturbation $\dot{y}=g=f+\sigma\hat{\varphi}$ has $\left\Vert f-g\right\Vert_1 < \epsilon$ and has no fixed points in $[a-\delta,b+\delta]$, which means it cannot be topologically equivalent to $\dot{x}=f$ in virtue of \eqref{Eq: TE}.  On the other hand, for subcase (ii), the perturbation $\dot{y}=g$, where
\begin{equation*}
g:=f+\sigma\hat{\vartheta},
\end{equation*}
with
\begin{equation}
\hat{\vartheta}(x):=-\frac{2\hat{\varphi}(x)}{e\left(  a+b\right)  }\left(x-\frac{a+b}{2}\right),
\label{Eq: case3v}
\end{equation}
produces an arbitrarily small $C^{1}$-perturbation of $\dot{x}=f$ having precisely three hyperbolic fixed points in $[a-\delta,b+\delta]$.  Therefore, $f$ is not structurally stable for any of these subcases.\\

\noindent \textbf{Case 4}: Suppose $\dot{x}=f$ has distinct fixed points $x_1,\, x_2,\, x_3,\, \ldots$ and $x_n \rightarrow x_*$ , so that the limit  $x_*$ is a nonhyperbolic fixed point. The sequence and its limit may be assumed to lie in an open interval $J = (x_*-r,x_*+r)$ contained in $(0,1)$, which contains no other fixed points. The sequence might consist of all hyperbolic fixed points as shown in Fig \ref{Fig: countable}, or it might be comprised of some combination of hyperbolic fixed points and nonhyperbolic fixed points of the types treated in Cases 1 and 2.  Once again, for the purpose of demonstration, it may be useful for the reader to think of an example function
\begin{equation}
f(x) = \begin{cases}
f_\text{in}\\
f_\text{out}
\end{cases} = \begin{cases}
(x-x_*)\sin\left(\frac{1}{(x-x_*)^3}\right) & \text{for $x\in (x_* - \delta, x_* + \delta)$},\\
f_\text{out} & \text{for $x\in \mathbb{S}^1\setminus (x_* - \delta, x_* + \delta)$};
\end{cases}
\label{Eq: case4Func}
\end{equation}
where $f_\text{out}$ is an arbitrary function with properties satisfying the hypothesis of the theorem.
\begin{figure}[htbp]
\centering
\includegraphics[width=.9\textwidth]{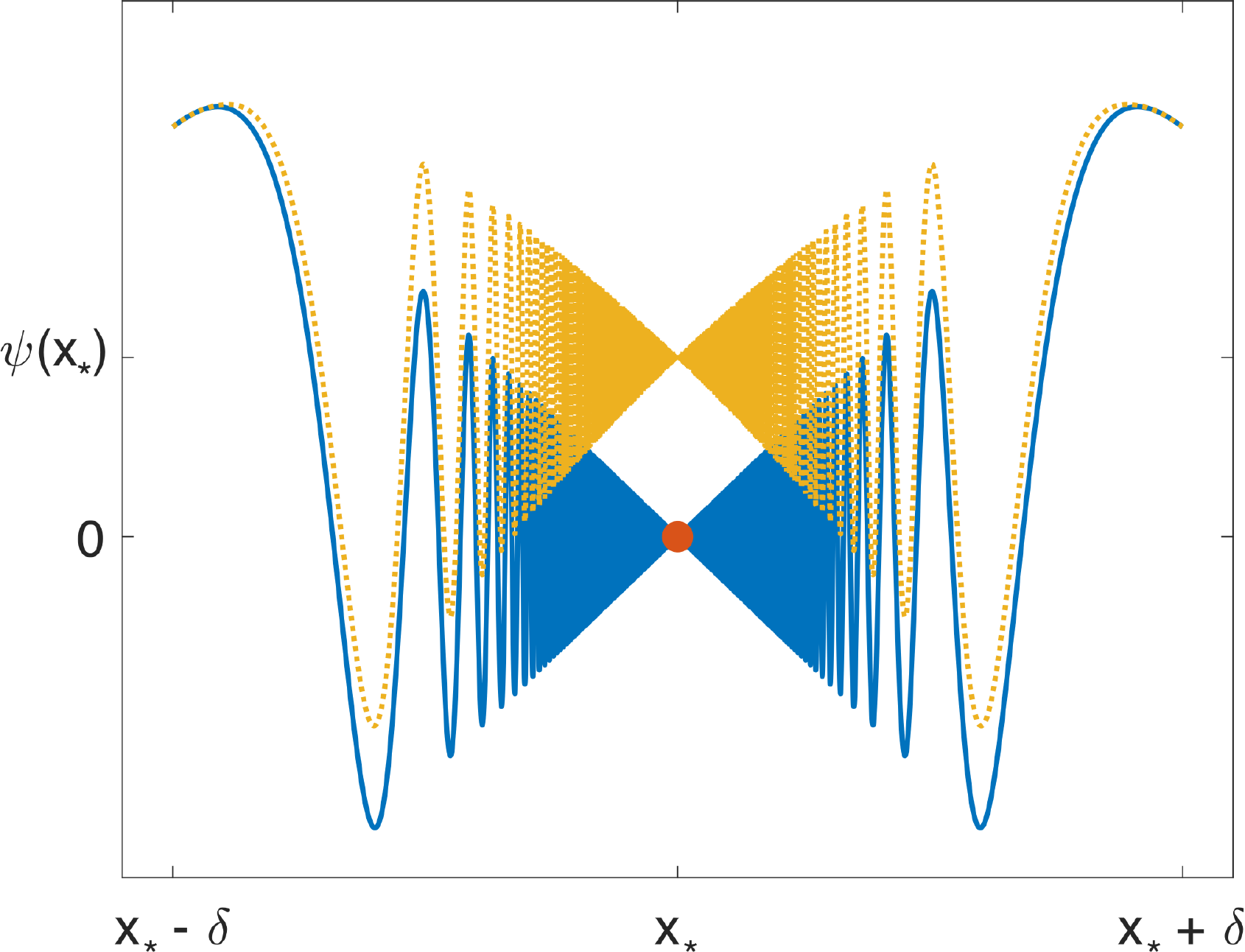}
\caption{An example \eqref{Eq: case4Func} of a function (bottom solid blue curve), $f$, with countably many fixed points (hyperbolic or nonhyperbolic) in a small interval, $(x_* - \delta, x_* + \delta)$, centered at a nonhyperbolic fixed point, $x_*$.  The function is being perturbed by a bump function restricted to the small interval, and the perturbed function (top dashed yellow curve) is shown to have lost infinitely many fixed points.}
\label{Fig: countable}
\end{figure}

Notice that $\{x_n\}$ is countably infinite, so that if we perturb the system $\dot{x}=f$ to a system $\dot{y}=g$ with only finitely many fixed points in $J$, and the same fixed points in the complement of $J$, the two systems must be topologically inequivalent.  Since $f \in C^{1}$, for any $\epsilon>0$ there is a positive $\delta=\delta(\epsilon)<\epsilon$ such that $|x-x_*| < \delta$ implies $\left\vert f(x)\right\vert =|f(x)-f(x_*)| < \epsilon\left\vert x-x_*\right\vert$. Furthermore, there are only finitely many fixed points in $J \smallsetminus B_s(x_*)$ for any $0<s<r$. Let us use the bump function
\begin{equation}
\psi(x):=
\begin{cases}
\exp\left(  \frac{-(r/2)^{2}}{(r/2)^{2}-(x-x_{\ast})^{2}}\right) & \text{for}\;x\in(x_{\ast}-r/2,x_{\ast}+r/2),\\
0 & \text{for}\;x\notin(x_{\ast}-r/2,x_{\ast}+r/2);
\end{cases},
\end{equation}
and for any given $\epsilon>0$ choose $\sigma>0$ such that $\left\Vert g-f\right\Vert_1=\left\Vert \left(f+\sigma\psi\right)-f\right\Vert_1<\epsilon$.  Then $g$ has no zeros in $B_s(x_*)$ for some $0<s<r/2$ and so only finitely many fixed points in $J$, which means that $f$ is not structurally stable. Thus, the proof of Peixoto's theorem on $\mathbb{S}^1$ is complete.
\end{proof}

We note that owing to the compactness of the unit interval with identified end points, it suffices in Theorem 1 to simply require all fixed points be hyperbolic.

\begin{exmp}[Combining cases]
In this short example let us consider the vector field
\begin{equation}
f(x) = \begin{cases}
r(x) & \text{for $x \in [0,1/4]$},\\
s(x) & \text{for $x \in (1/4, 1/2)$},\\
0 & \text{for $x \in [1/2,1]$};
\end{cases}
\label{Eq:  Combination Vector Field}
\end{equation}
where
\begin{equation*}
r(x) = x^5\sin(1/x^3);
\end{equation*}
and
\begin{equation*}
s(x) = 
r(1/4)\begin{cases}
\exp\left(1 - \left(\frac{1}{4}\right)^{2}\middle/\left[\left(\frac{1}{4}\right)^{2}-\left(x-\frac{1}{4}\right)^{2}\right]\right)  & \text{for}\;x\in(0,1/2),\\
0 & \text{for}\;x\notin(0,1/2);
\end{cases}
\end{equation*}
This is a combination of Cases 3 and 4.  On the first interval $[0,1/4]$, the function crosses the x-axis countably infinite times such as in Case 4.  The last interval $[1/2, 1]$ is like Case 3 where we have an interval of fixed points; that is, an uncountable amount of them.  Finally, the middle interval transitions between the functions in the first and last intervals smoothly using the Bump function.  The function $f(x)$ is illustrated in Fig. \ref{Fig:  Combination Example}(a).

Now we set up the perturbation.  We recall from the proof that a bump function is able to keep the perturbation on a bounded interval thereby leaving the rest of the function unaffected.  Here we have two intervals on which we would need to introduce a perturbation to show sufficiency of Theorem \ref{Thm: peixoto}.  A combination of two bump functions centered within the two intervals and zero everywhere else will work.  Consider the perturbation
\begin{equation}
\eta(x) = \epsilon\begin{cases}
\exp\left(- \left(\frac{1}{8}\right)^{2}\middle/\left[\left(\frac{1}{8}\right)^{2}-\left(x-\frac{1}{8}\right)^{2}\right]\right)  & \text{for}\;x\in[0,1/4],\\
\exp\left(- \left(\frac{1}{4}\right)^{2}\middle/\left[\left(\frac{1}{4}\right)^{2}-\left(x-\frac{3}{4}\right)^{2}\right]\right)  & \text{for}\;x\in[1/2,1],\\
0 & \text{otherwise}
\end{cases}
\label{Eq:  Combination Perturbation}
\end{equation}
for any sufficiently small $\epsilon > 0$.  The combined bump functions are shown in Fig. \ref{Fig:  Combination Example}(b).

After setting up our perturbation we simply add it to the original function:  $g(x) = f(x) + \eta(x)$.  We notice in Fig. \ref{Fig:  Combination Example}(c) that $g(x)$ will now have finitely many zeros; that is, $\dot{x} = g(x)$ has finitely many fixed points.  Therefore, $\dot{x} = g(x)$ and $\dot{x} = f(x)$ are not topologically equivalent, and hence the sufficiency criteria for Theorem \ref{Thm: peixoto} is satisfied.
\begin{figure}[htbp]
\centering
\stackinset{r}{1mm}{b}{6mm}{\textbf{\large (a)}}{\includegraphics[width = .32\textwidth]{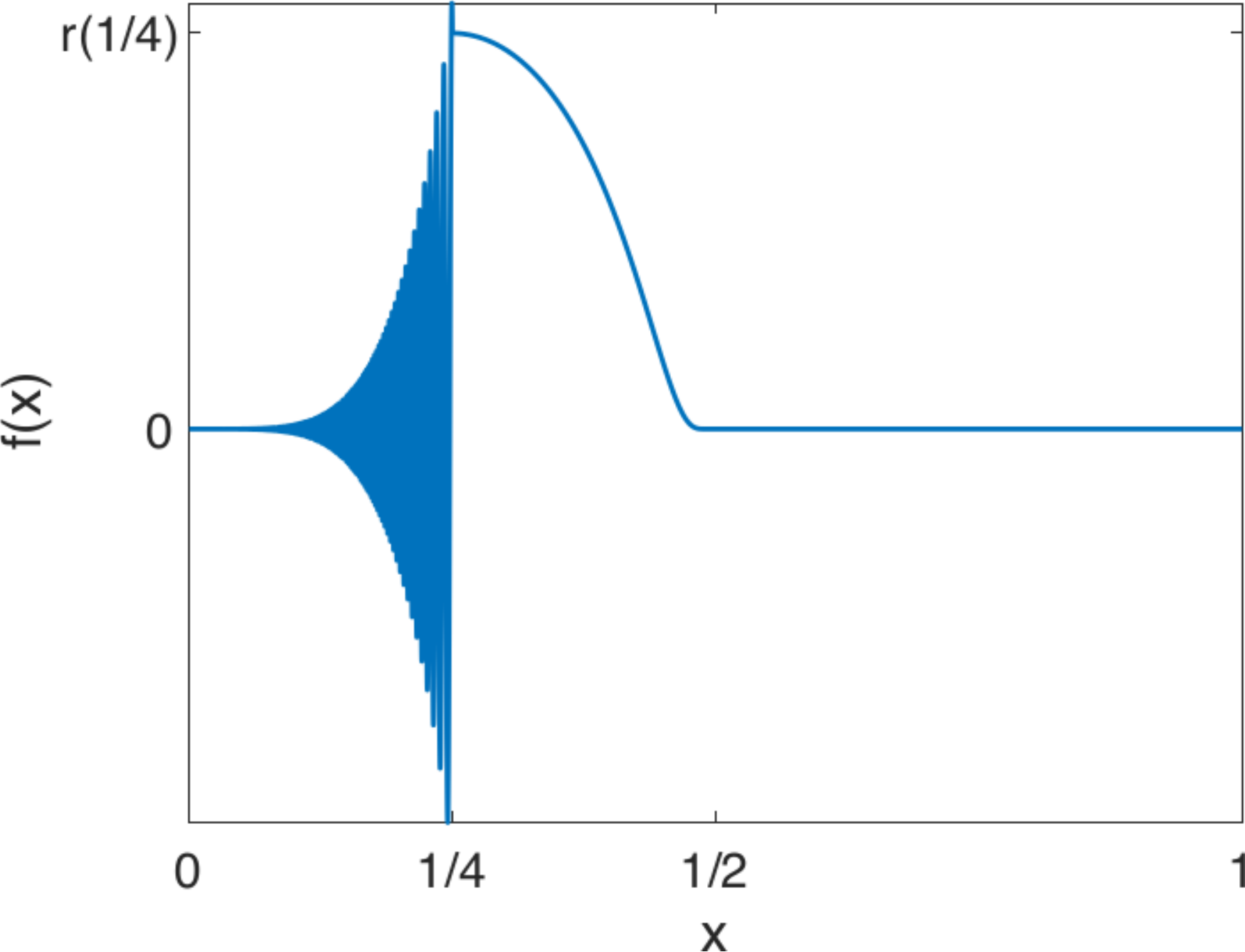}}
\stackinset{r}{1mm}{b}{6mm}{\textbf{\large (b)}}{\includegraphics[width = .32\textwidth]{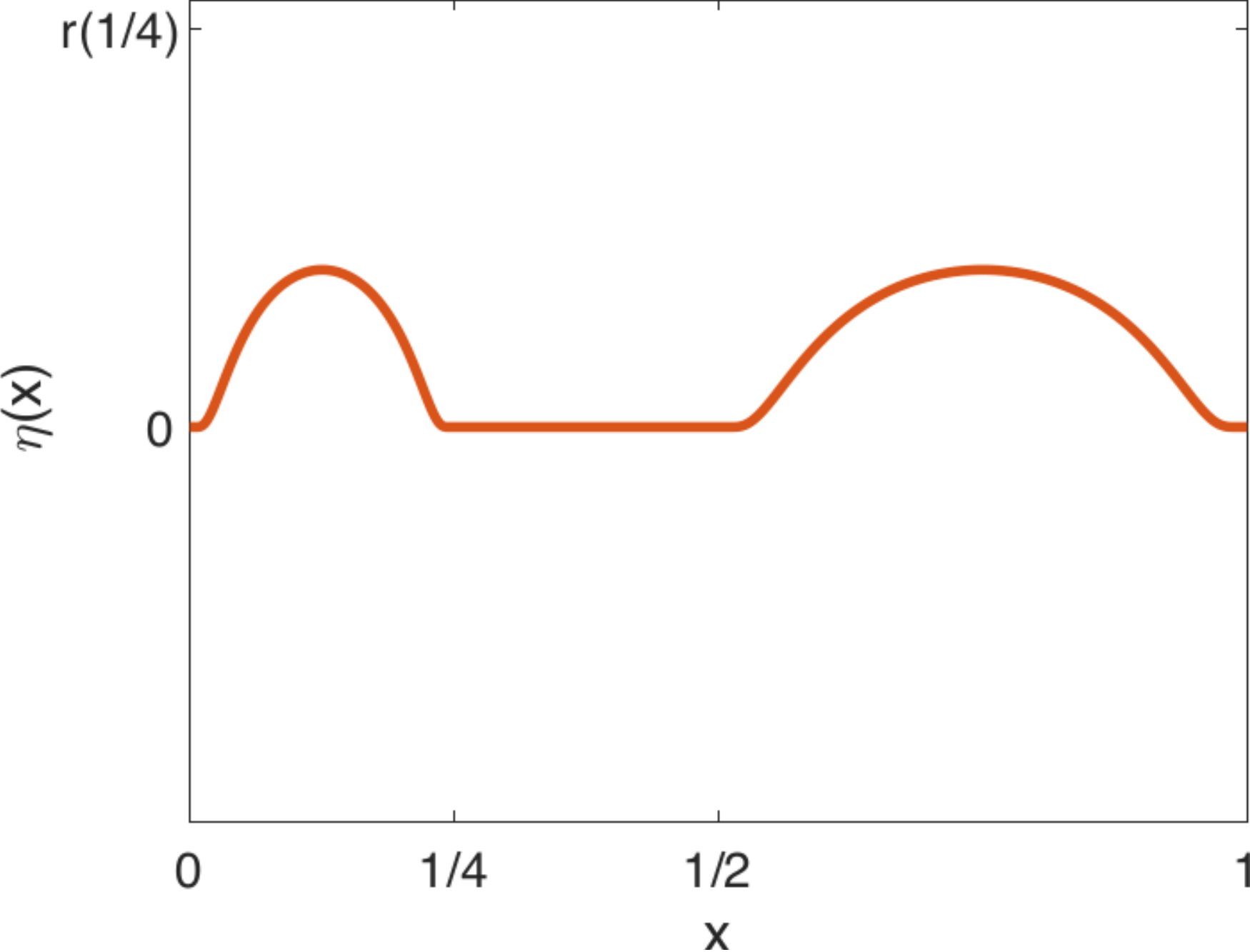}}
\stackinset{r}{1mm}{b}{6mm}{\textbf{\large (c)}}{\includegraphics[width = .32\textwidth]{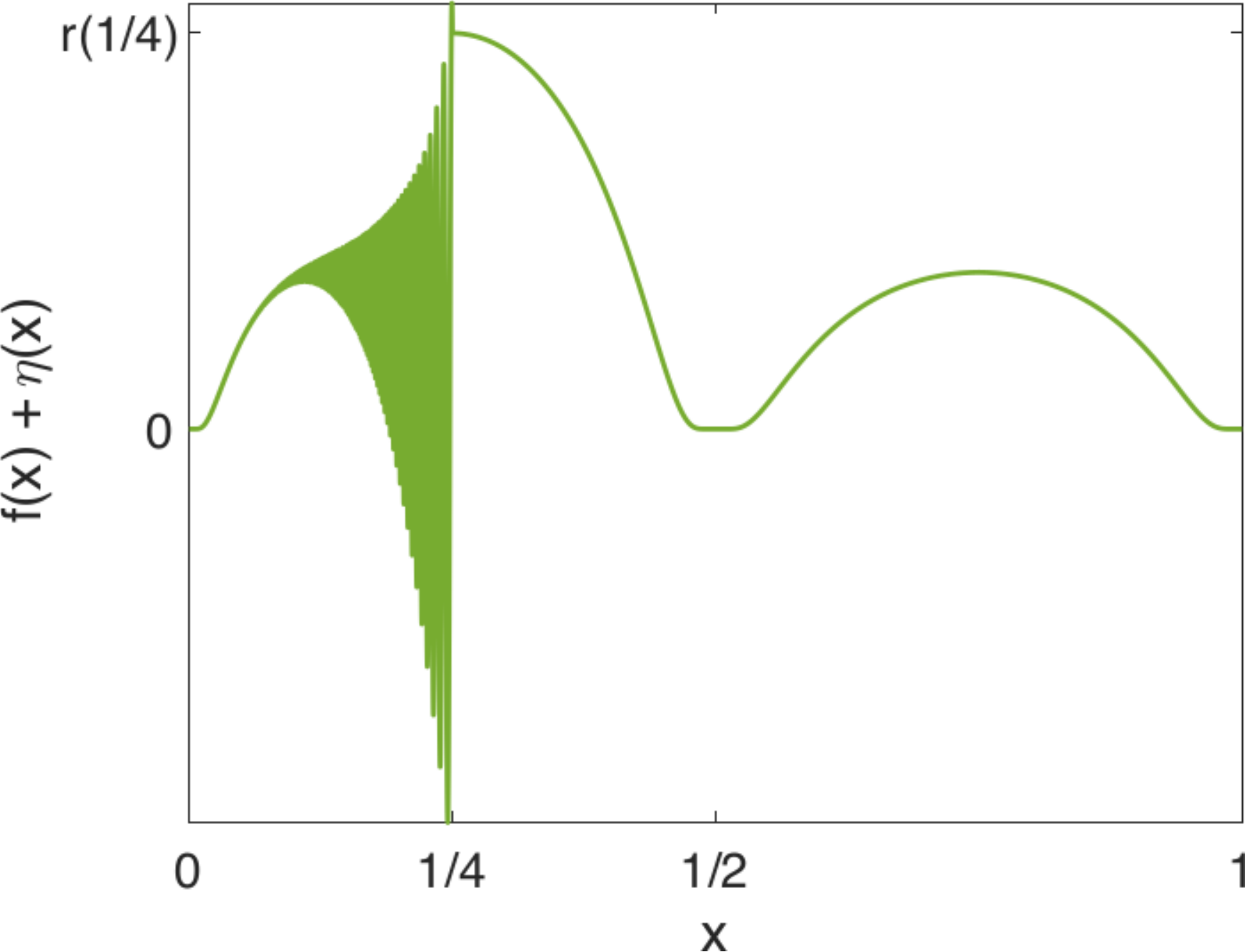}}
\caption{Combination of Cases 3 and 4.  \textbf{(a)} Plot of the function $f(x)$ \eqref{Eq:  Combination Vector Field}.  \textbf{(b)}  Plot of the perturbation $\eta(x)$ \eqref{Eq:  Combination Perturbation}.  \textbf{(c)}  Plot of the perturbed function $g(x) = f(x) + \eta(x)$.}
\label{Fig:  Combination Example}
\end{figure}
\end{exmp}

\section{Density theorem on $\mathbb{S}^{1}$}\label{Sec: Density}

We now prove the one-dimensional analog of the density part of Theorem \ref{Thm: Peixoto2D}. It is convenient to introduce the following notation towards this end.  Define $SS^{1}\left(\mathbb{S}^{1}\right)$ to be the $C^{1}$-structurally stable systems $\dot{x}=f(x)$ on $C^{1}\left(\mathbb{S}^{1}\right).$

\begin{thm}
The set of dynamical systems $SS^{1}\left(\mathbb{S}^{1}\right)$ is $C^{1}$ open and dense in $C^{1}\left(\mathbb{S}^{1}\right).$
\label{Thm: Density}
\end{thm}

\begin{proof}

The openness follows directly from the necessity proof of Theorem \ref{Thm: peixoto}, and the density is essentially a straightforward consequence of the sufficiency argument for the same theorem. In particular, it was shown in the necessity proof that the fixed point hypothesis is preserved under sufficiently small perturbations, and so $SS^{1}\left(\mathbb{S}^{1}\right)$ is a $C^{1}$ open subset of $C^{1}\left(\mathbb{S}^{1}\right)$.

It is clear from the methods used in proving the sufficiency part of Theorem \ref{Thm: peixoto}, that for any $C^{1}$ dynamical system $\dot{x}=f(x)$ on $\mathbb{S}^{1}$ there is an arbitrarily small $C^{1}$ perturbation $\dot{y}=g(y)$ such that $g(y)$ has only finitely many zeros. Then, using the bump function methods employed for Cases 1 and 2 of the sufficiency portion in Theorem \ref{Thm: peixoto}, we can obtain a further arbitrarily small perturbation $\dot{z}=h(z)$ with only hyperbolic fixed points, thereby completing the proof.

\end{proof}

\section{Brief aside}\label{Sec: Remarks}

It is worth noting that one could have used several other types of bump function based perturbations in the above proofs of the necessity of the hyperbolic hypothesis in Theorem \ref{Thm: peixoto} and the density result in Theorem \ref{Thm: Density}. For example, the functions $\vartheta$ and $\hat{\vartheta}$ used for Case 2 and Case 3 (ii), respectively, in the necessity proof of Theorem \ref{Thm: peixoto} could be replaced with an appropriate form of the derivative of a bump function, as is evident from the plot of the first and second derivatives of the simple bump function \eqref{Eq: simple bump} given in Fig \ref{Fig: dbump}.

\begin{figure}[htb]
\centering
\stackinset{l}{8mm}{b}{10mm}{\textbf{\large (a)}}{\includegraphics[width=.45\textwidth]{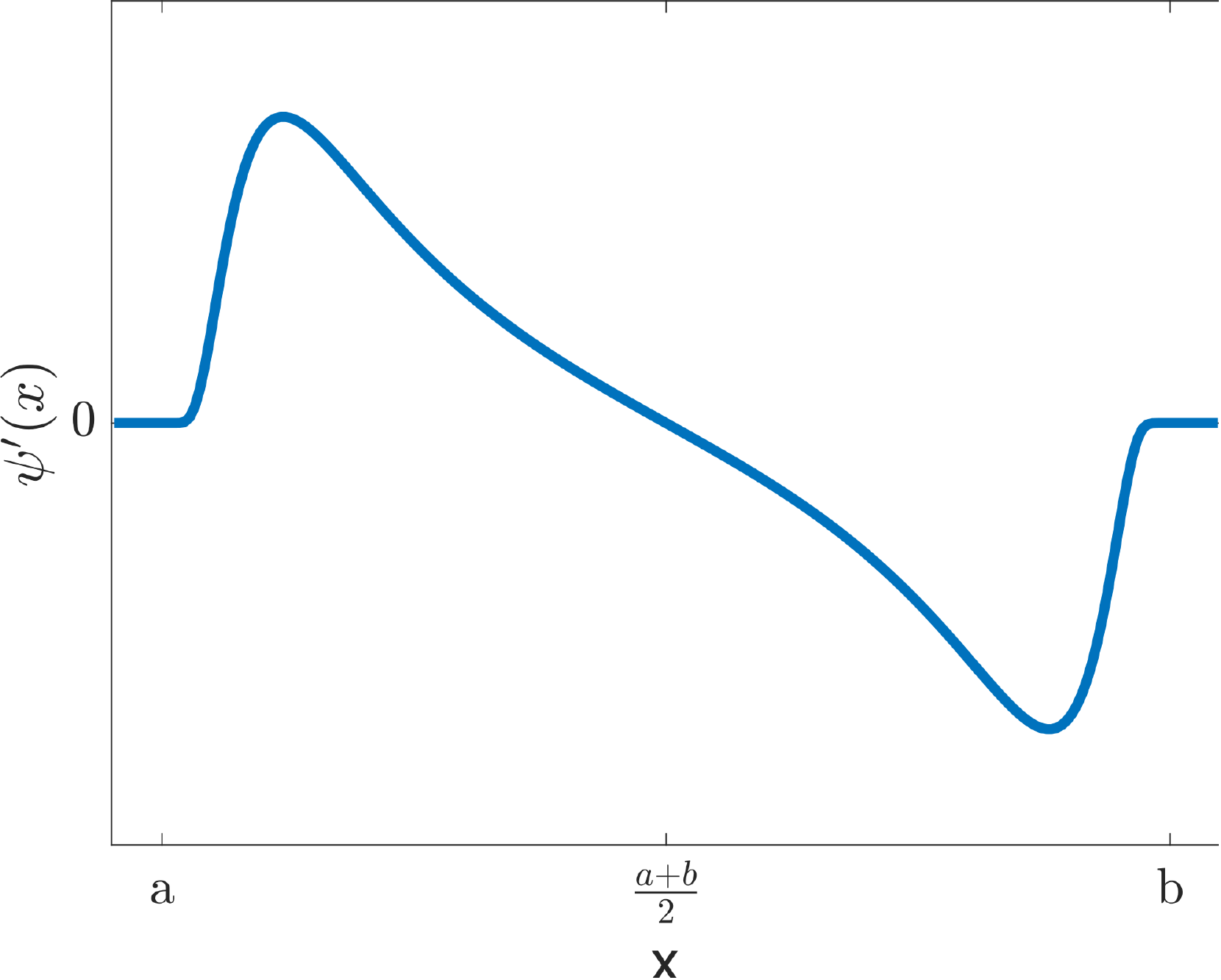}}\qquad
\stackinset{l}{8mm}{b}{10mm}{\textbf{\large (b)}}{\includegraphics[width=.45\textwidth]{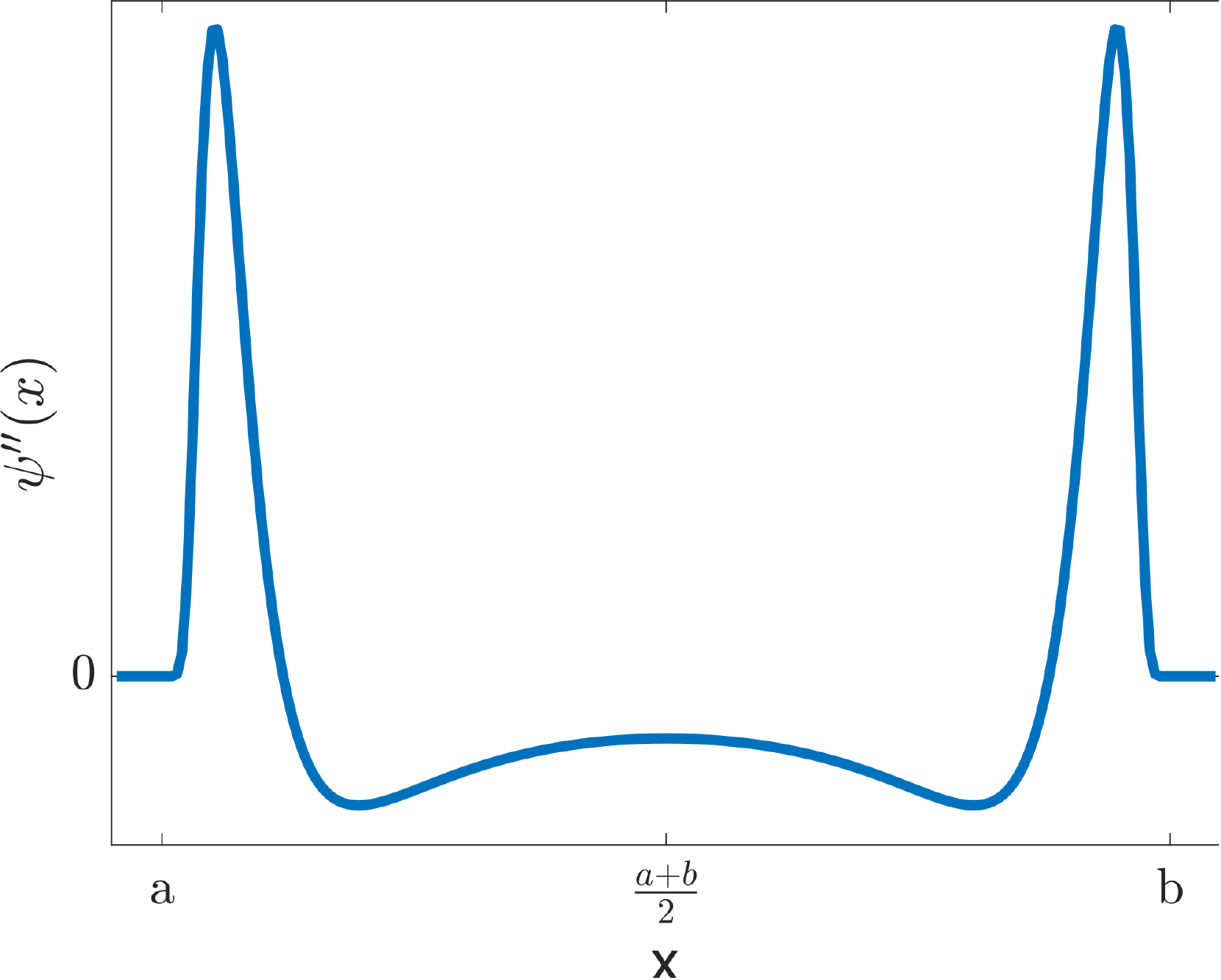}}
\caption{The \textbf{(a)} first and \textbf{(b)} second derivatives of the simple bump function.}
\label{Fig: dbump}
\end{figure}

Finally, it is interesting to remark that the proofs of both Theorems \ref{Thm: peixoto} and \ref{Thm: Density} can be reduced to just a few lines by the application of a standard transversality theorem described for example in \cite{Guillemin-Pollack74}, which is an indication of the importance of differential topology in the modern theory of dynamical systems.  However, the authors also note that this is not a standard topic of traditional undergraduate mathematics.

\section{Conclusion and Memoriam}\label{Sec: Conclusion}

The epochal structural stability and density theorems of Peixoto for dynamical systems on closed surfaces have long and complicated proofs involving concepts unfamiliar to many undergraduate enthusiasts. In this manuscript we have demonstrated that the one-dimensional analogs of these theorems can be proved using methods that are well known to most undergraduate mathematics majors, thus providing a useful introduction to many of the elements of the two-dimensional proofs. One might well imagine that the Peixotos themselves considered the one-dimensional version and used it, along with the pioneering efforts of Andronov and Pontryagin, as a guide for their theorems.

In memory of Mar\'{i}lia and Maur\'{i}cio Peixotos' lives on this 100\textsuperscript{th} anniversary of their births, the authors would like to express the personal significance of the Peixotos' work and Peixoto's theorem in particular.  For A.R., learning about Peixoto's theorem in D.B.'s undergraduate course on nonlinear dynamics was the first time the study of dynamical systems went from a curiosity about its applications to a voracious interest in the underlying analysis.  Furthermore, Maur\'{i}cio Peixoto's friendship with Solomon Lefschetz also resonated with A.R., as he was also inspired by his close relationship with his mentor D.B.  This is of particular significance since the recent untimely passing of D.B. D.B's contributions to Dynamical Systems cannot be overstated, and he will be deeply missed by all that knew him.

%%%%%%%%%%%%%%%%%%%%

\section{Acknowledgement}

The authors would like to thank SIAM DS Web as the original module was conceived as part of the SIAM DS Web 2013 pedagogy prize.  Moreover, the authors would also like to show their appreciation towards the reviewers for their detailed feedback and suggestions that were instrumental in improving this manuscript.  In particular, we appreciate one of the reviewers for suggesting the function explored in Example 1.  A.R. appreciates the support of the Department of Applied Mathematics at UW, and D.B. appreciates the support of the Department of Mathematical Sciences at NJIT.  Finally, A.R. would like to express his heartfelt gratitude to his late mentor and friend D.B.

%%%%%%%%%%%%%%%%%%%%

\bibliographystyle{unsrt}
\bibliography{Peixoto}

\begin{thebibliography}{10}

\bibitem{Leibniz}
Gottfried Wilhelm~von Leibniz.
\newblock Nova methodus pro maximis et minimis, itemque tangentibus, quae nec
  fractas nec irrationales quantitates moratur, et singulare pro illis calculi
  genus.
\newblock {\em Acta eruditorum}, 1684.

\bibitem{Principia}
Isaac Newton.
\newblock {\em Philosophi{\ae} Naturalis Principia Mathematica}.
\newblock The Royal Society, 1687.

\bibitem{Methodus}
Isaac Newton.
\newblock Methodus fluxionum et serierum infinitarum cum eisudem applicatione
  ad curvarum geometriam.
\newblock {\em Opuscola mathematica philosophica et philologica}, 1744.

\bibitem{Bernoulli}
Jacob Bernoulli.
\newblock Explicationes, annotationes et additiones ad ea, quae in actis sup.
  de curva elastica, iisochrona paracentrica, et velaria, hinc inde memorata,
  et paratim controversa legundur; ubi de linea mediarum directionum, alliisque
  novis.
\newblock {\em Acta Eruditorum}, 1695.

\bibitem{Poincare1}
Jules~Henri Poincar{\'e}.
\newblock {\em Les m{\'e}thodes nouvelles de la m{\'e}canique c{\'e}leste},
  volume~1.
\newblock Gauthier-Villars, Paris, France, 1892-1899.

\bibitem{Poincare2}
Jules~Henri Poincar{\'e}.
\newblock {\em Le{\c c}ons de M{\'e}canique C{\'e}leste}, volume~1.
\newblock Gauthier-Villars, 1905-1910.

\bibitem{Strogatz94}
Steven Strogatz.
\newblock {\em Nonlinear Dynamics and Chaos}.
\newblock Westview Press, Cambridge, MA, 1994.

\bibitem{Perko01}
L.~Perko.
\newblock {\em Differential Equations and Dynamical Systems}.
\newblock Springer-Verlag, New York, NY, 3 edition, 2001.

\bibitem{Meiss07}
James Meiss.
\newblock {\em Differential Dynamical Systems}.
\newblock SIAM, Philadelphia, PA, 2007.

\bibitem{BlanchardDevaneyHallODE}
P.~Blanchard, R.~L. Devaney, and G.~R. Hall.
\newblock {\em Differential Equations}, volume~4.
\newblock Cengage Learning, 2011.

\bibitem{Smale1967}
Stephen Smale.
\newblock Differentiable dynamical systems.
\newblock {\em Bull. Amer. Math. Soc.}, 73:747--817, 1967.

\bibitem{Milnor97}
J.~Milnor.
\newblock {\em Topology from the Differentiable Viewpoint}.
\newblock Princeton University Press, Princeton, NJ, 1997.

\bibitem{Sotomayor2001}
J.~Sotomayor.
\newblock Introduction: A few words about mauricio m. peixoto on his
  80\textsuperscript{th} birthday.
\newblock {\em Comput. Appl. Math.}, 20(1):1--5, 2001.

\bibitem{PeixotoPeixoto1959}
Marilia~Chaves Peixoto and Maur{\'\i}cio~Matos Peixoto.
\newblock Structural stability in the plane with enlarged boundary conditions.
\newblock {\em An. Acad. Bras. Cienc.}, 31:135--160, 1959.

\bibitem{Peixoto62}
Maur{\'\i}cio~Matos Peixoto.
\newblock Structural stability on two-dimensional manifolds.
\newblock {\em Topology}, 1:101--120, 1962.

\bibitem{Andronov-Pontryagin37}
A.~Andronov and L.~Pontryagin.
\newblock Syst{\`e}mes grossiers.
\newblock {\em Dokl. Akad. Nauk. SSSR}, 14:247--251, 1937.

\bibitem{Peixoto1959}
Maur{\'\i}cio~Matos Peixoto.
\newblock On structural stability.
\newblock {\em Ann. of Math}, 69(2):199--222, 1959.

\bibitem{Guillemin-Pollack74}
V.~Guillemin and A.~Pollack.
\newblock {\em Differential Topology}.
\newblock Prentice-Hall, 1974.

\end{thebibliography}

\end{document}